\documentclass{amsart}

\usepackage{amsfonts,amssymb,verbatim,amsmath,amsthm,latexsym,textcomp,amscd}
\usepackage{latexsym,amsfonts,amssymb,epsfig,verbatim}
\usepackage{amsmath,amsthm,amssymb,latexsym,graphics,textcomp}
\usepackage{graphicx}
\graphicspath{ {./sketches/} }
\usepackage{url}
\usepackage{psfrag}
\usepackage{amsmath,amsthm,amsfonts,amssymb,mathrsfs,bm,graphicx,stmaryrd}
\usepackage{mathtools}
\usepackage[usenames]{color}
\usepackage[linktocpage]{hyperref}
\usepackage[latin1]{inputenc}
\usepackage{microtype}

\usepackage{tikz}
\usetikzlibrary{shapes,arrows}

\begin{document}

\tikzstyle{decision} = [diamond, draw, fill=gray!20, 
    text width=4.5em, text badly centered, node distance=3cm, inner sep=0pt]
\tikzstyle{block} = [rectangle, draw, fill=gray!20, 
    text width=10em, text centered, rounded corners, minimum height=3em]
\tikzstyle{line} = [draw, -latex']
\tikzstyle{cloud} = [draw, ellipse,fill=gray!20, node distance=3cm,
    minimum height=2em]

\newtheorem{thm}{Theorem}[section]
\newtheorem{prop}[thm]{Proposition}
\newtheorem{assume}[thm]{Assumption}
\newtheorem{lemma}[thm]{Lemma}
\newtheorem{cor}[thm]{Corollary}
\newtheorem{conj}[thm]{Conjecture}
\newtheorem{claim}[thm]{Claim}
\newtheorem{qn}[thm]{Question}
\newtheorem{defn}[thm]{Definition}
\newtheorem{defth}[thm]{Definition-Theorem}
\newtheorem{obs}[thm]{Observation}
\newtheorem{rmk}[thm]{Remark}
\newtheorem{ans}[thm]{Answers}
\newtheorem{slogan}[thm]{Slogan}
\newtheorem{propn}[thm]{Proposition}
\newtheorem{propt}[thm]{Property}
\newtheorem{ex}[thm]{Example}

\newcommand{\hhat}{\widehat}
\newcommand{\C}{{\mathbb C}}
\newcommand{\B}{{\mathbb B}}
\newcommand{\Ga}{{\Gamma}}
\newcommand{\G}{{\Gamma}}
\newcommand{\s}{{\Sigma}}
\newcommand{\PSL}{{PSL_2 (\mathbb{C})}}
\newcommand{\pslc}{{PSL_2 (\mathbb{C})}}
\newcommand{\pslr}{{PSL_2 (\mathbb{R})}}
\newcommand{\Gr}{{\mathcal G}}
\newcommand{\integers}{{\mathbb Z}}
\newcommand{\natls}{{\mathbb N}}
\newcommand{\ratls}{{\mathbb Q}}
\newcommand{\reals}{{\mathbb R}}
\newcommand{\proj}{{\mathbb P}}
\newcommand{\lhp}{{\mathbb L}}
\newcommand{\tube}{{\mathbb T}}
\newcommand{\cusp}{{\mathbb P}}
\newcommand\AAA{{\mathcal A}}
\newcommand\HHH{{\mathbb H}}
\newcommand\BB{{\mathcal B}}
\newcommand\CC{{\mathcal C}}
\newcommand\DD{{\mathcal D}}
\newcommand\EE{{\mathcal E}}
\newcommand\FF{{\mathcal F}}
\newcommand\GG{{\mathcal G}}
\newcommand\HH{{\mathcal H}}
\newcommand\II{{\mathcal I}}
\newcommand\JJ{{\mathcal J}}
\newcommand\KK{{\mathcal K}}
\newcommand\LL{{\mathcal L}}
\newcommand\MM{{\mathcal M}}
\newcommand\NN{{\mathcal N}}
\newcommand\OO{{\mathcal O}}
\newcommand\PP{{\mathcal P}}
\newcommand\QQ{{\mathcal Q}}
\newcommand\RR{{\mathcal R}}
\newcommand\SSS{{\mathcal S}}
\newcommand\TT{{\mathcal T}}
\newcommand\UU{{\mathcal U}}
\newcommand\VV{{\mathcal V}}
\newcommand\WW{{\mathcal W}}
\newcommand\XX{{\mathcal X}}
\newcommand\YY{{\mathcal Y}}
\newcommand\ZZ{{\mathcal Z}}
\newcommand{\iid}{{i.i.d.\ }}
	\renewcommand{\ae}{{a.e.\ }}
\newcommand\CH{{\CC\Hyp}}
\newcommand{\Chat}{{\hat {\mathbb C}}}
\newcommand\MF{{\MM\FF}}
\newcommand\PMF{{\PP\kern-2pt\MM\FF}}
\newcommand\ML{{\MM\LL}}
\newcommand\PML{{\PP\kern-2pt\MM\LL}}
\newcommand\GL{{\GG\LL}}
\newcommand\Pol{{\mathcal P}}
\newcommand\half{{\textstyle{\frac12}}}
\newcommand\Half{{\frac12}}
\newcommand\Mod{\operatorname{Mod}}
\newcommand\Area{\operatorname{Area}}
\newcommand\ep{\epsilon}
\newcommand\Hypat{\widehat}
\newcommand\Proj{{\mathbf P}}
\newcommand\U{{\mathbf U}}
 \newcommand\Hyp{{\mathbf H}}
\newcommand\D{{\mathbf D}}
\newcommand\Z{{\mathbb Z}}
\newcommand\R{{\mathbb R}}
\newcommand\Q{{\mathbb Q}}
\newcommand\E{{\mathbb E}}
\newcommand\EXH{{ \EE (X, \HH_X )}}
\newcommand\EYH{{ \EE (Y, \HH_Y )}}
\newcommand\GXH{{ \GG (X, \HH_X )}}
\newcommand\GYH{{ \GG (Y, \HH_Y )}}
\newcommand\ATF{{ \AAA \TT \FF }}
\newcommand\PEX{{\PP\EE  (X, \HH , \GG , \LL )}}
\newcommand{\lct}{\Lambda_{CT}}
\newcommand{\lel}{\Lambda_{EL}}
\newcommand{\lgel}{\Lambda_{GEL}}
\newcommand{\lre}{\Lambda_{\mathbb{R}}}

\newcommand\supp{\operatorname{supp}}
\newcommand\til{\widetilde}
\newcommand\length{\operatorname{length}}
\newcommand\tr{\operatorname{tr}}
\newcommand\cone{\operatorname{cone}}
\newcommand\gesim{\succ}
\newcommand\lesim{\prec}
\newcommand\simle{\lesim}
\newcommand\simge{\gesim}
\newcommand{\simmult}{\asymp}
\newcommand{\simadd}{\mathrel{\overset{\text{\tiny $+$}}{\sim}}}
\newcommand{\ssm}{\setminus}
\newcommand{\diam}{\operatorname{diam}}
\newcommand{\pair}[1]{\langle #1\rangle}
\newcommand{\T}{{\mathbf T}}
\newcommand{\I}{{\mathbf I}}
\newcommand{\pG}{{\partial G}}
 \newcommand{\oxi}{{[1,\xi)}}
\newcommand{\cg}{\mathcal{G}}

\newcommand{\gr}{\operatorname{Gr}}
\newcommand{\fl}{\operatorname{Fl}}
\newcommand{\grkn}{{\gr}_{k,n}}
\newcommand{\grkinf}{{\gr}_{k,\infty}}
\newcommand{\grknc}{\grkn(\C)}
\newcommand{\grkinfc}{\grkinf(\C)}
\newcommand{\tw}{\operatorname{tw}}
\newcommand{\base}{\operatorname{base}}
\newcommand{\diag}{\operatorname{\bf{diag}}}
\newcommand{\rest}{|_}
\newcommand{\bbar}{\overline}
\newcommand{\lbar}{\underline}
\newcommand{\UML}{\operatorname{\UU\MM\LL}}
\newcommand{\EL}{\mathcal{EL}}
\newcommand{\ncox}{{N_C([o,\xi))}}
\newcommand{\qle}{\lesssim}
\renewcommand{\Bbb}{\mathbb}

\def\ind{{\mathrm{ind}}}
\def\N{\mathbb{N}}
\def\L{\mathbb{L}}
\def\P{\mathbb{P}}
\def\Z{\mathbb{Z}}
\def\R{\mathbb{R}}
\def\S{\mathcal{S}}
\def\X{\mathcal{X}}
\def\E{\mathbb{E}}
\def\H{\mathbb{H}}
\def\l{\ell}
\def\Ups{\Upsilon}
\def\om{\omega}
\def\Om{\Omega}

\newcommand\Gomega{\Omega_\Gamma}
\newcommand\nomega{\omega_\nu}
\newcommand\omegap{{(\Omega,\P)}}
\newcommand\omegapp{{(\Omega',\P)}}

\DeclarePairedDelimiter\floor{\lfloor}{\rfloor}
\newcommand{\cF}{\mathcal{F}}
\newcommand{\cD}{\mathcal{D}}
\newcommand{\cZ}{\mathcal{Z}}
\newcommand{\cf}{\mathcal{F}}
\newcommand{\cA}{\Omega}
\newcommand{\cB}{\mathcal{B}}
\newcommand{\cC}{\mathcal{C}}
\newcommand{\cT}{\mathcal{T}}
\newcommand{\rb}{\mathfrak{A}}
\newcommand{\Var}{\hbox{Var}}
\newcommand{\ee}{\hbox{e}_1}
\newcommand{\cd}{\mathcal{D}}
\newcommand{\ce}{\mathcal{E}}

\newcommand{\bcomment}[1]{\textcolor{blue}{#1}}
\newcommand{\boundary}{\partial}
\newcommand{\dpt}{\mathrm{dpt}}

\title{The Gromov-Tischler theorem for stratified spaces}

\date{\today}

\author{Mahan Mj}
\address{School
	of Mathematics, Tata Institute of Fundamental Research. 1, Homi Bhabha Road, Mumbai-400005, India}

\email{mahan@math.tifr.res.in}

\author{Balarka Sen}
\address{School
	of Mathematics, Tata Institute of Fundamental Research. 1, Homi Bhabha Road, Mumbai-400005, India}
 \email{balarka@math.tifr.res.in}

\thanks{MM is supported by the Department of Atomic Energy, Government of India, under project no.12-R\&D-TFR-5.01-0500;
	and in part by a DST JC Bose Fellowship,  and an endowment from the Infosys Foundation.
}

\date{\today}

\begin{abstract}
We define a notion of a symplectic structure on stratified spaces, and demonstrate that given a symplectic structure on a stratified space $X$ with integral cohomology class, $X$ can be symplectically embedded in some complex projective space equipped with the standard K\"ahler form. This extends a theorem, due to Gromov and Tischler for manifolds, to stratified spaces.
\end{abstract}

\maketitle

\tableofcontents

\section{Introduction}\label{sec-intro}

 A celebrated theorem due to Gromov \cite[Section 3.4.2]{Gromov_PDR}, \cite{gromov-icm70} and Tischler \cite{tischler} says the following:

\begin{thm}\label{thm-gt-mfld}
Let $(M,\om)$ be a closed symplectic manifold with $[\omega] \in H^2(M; \Bbb Z)$. Then there exists $n \in \natls$ and an embedding $f : M \to \mathbb{CP}^n$ such that $f^*\omega_n = \omega$, where $\omega_n$ is the standard K\"{a}hler form on $\mathbb{CP}^n$.
\end{thm}

In  this paper, we first generalize Theorem \ref{thm-gt-mfld} to the context of Whitney 
stratified spaces (see Theorem \ref{thm-gt}). But at the outset, we need to develop the necessary tools to answer the following basic question due to Gromov
\cite[p. 343]{Gromov_PDR}:

\begin{qn}\label{qn-gromov}
	Can one define singular symplectic (sub) varieties? 
\end{qn}

Answers to Question \ref{qn-gromov} have appeared in the context of singular complex varieties and symplectic reduction\cite{SL_stratsympred,DJ_singsymp}.  We develop the general theory of symplectic stratified spaces here from a symplectic geometry viewpoint (as opposed to the more algebraic Poisson structure perspective of \cite{sjamaar-lms,sjamaar-rednquant}). The main theorem of this paper is the following (see Theorem \ref{thm-gt} for a more general statement):

\begin{thm}\label{thm-gt-intro}  Let $(X, \omega)$ be a compact stratified symplectic space where $\omega \in \Omega^2(X)$ is  integral.  There exists $N \geq 1$ and an embedding $f : X \to \mathbb{CP}^N$ such that $f^*\omega_N = \omega$, where $\omega_N$ is the standard K\"{a}hler form on $\mathbb{CP}^N$. 
\end{thm}

There are two basic tools we use to prove Theorem \ref{thm-gt-intro}:
\begin{enumerate}
\item a compression lemma (Proposition \ref{strat-compress}) that allows us to use a smooth stratified map to go between strata 
\item a variant of work of Verona (Proposition \ref{prop-che}) establishing a chain homotopy equivalence between the chain complex of smooth forms and the 
chain complex of controlled forms.
\end{enumerate}

A  construction going back to Nash (and refined by Gromov and Tischler)  can then be reworked in our context (Section \ref{sec-ngt}) to complete the proof.

Using Theorem \ref{thm-gt-intro}, we establish a relationship between the following notions (see Proposition \ref{prop-alimpliesgeo}):
\begin{enumerate}
\item The notion of a Poisson-symplectic stratified space introduced by Sjamaar-Lerman \cite{SL_stratsympred} in a Poisson algebra setup.
\item The dual geometric notion of a symplectic stratified space dealt with in this paper.
\end{enumerate} 
Proposition \ref{prop-alimpliesgeo} allows us to go from the algebra to the geometry using the inverse  of a matrix representation
of the symplectic form. This vindicates a hope expressed by Sjamaar-Lerman in \cite[p. 412]{SL_stratsympred} in the context of Cushman's conjecture \cite[Conjecture 6.14]{SL_stratsympred}
predicting universality of $\R^n$ for symplectic reductions:\\
``Quite possibly the symplectic embedding theorems of Gromov and Tischler  could be relevant in this context."\\
Proposition \ref{prop-geoimpliesal} in the converse direction allows us to go from the geometry to the algebra, using Theorem 
\ref{thm-gt-intro}. Here, we use a modified algebra of smooth functions $C^\infty (\OO p (X))$ consisting
of restrictions of smooth functions to a germ of an open neighborhood $\OO p (X)$ of $X$ in a projective space (provided by Theorem 
\ref{thm-gt-intro}).
While Cushman's conjecture  has turned out to be false in its strongest form, and approaches and partial positive results \cite{sjamaar-lms,lermann-expo} have utilized a primarily algebraic approach, Theorem \ref{thm-gt-intro} shows that a dual geometric statement is indeed true, if we replace $\R^n$ by
$\C P^n$.

Finally, we generalize a theorem of Gotay-Tuynman \cite{gotay-tuynman-sympred} to symplectic stratified spaces using a completely different argument from the one in
\cite{gotay-tuynman-sympred}. We use Theorem \ref{thm-gt-intro} in an essential way.

\begin{thm}[see Theorem \ref{gotay}]\label{gotay-intro} Let $(X, \omega)$ be a compact symplectic stratified space. Then there exists $n \geq 1$ and a Whitney stratified set $(Y, \Sigma) \subset \Bbb R^{2n}$ equipped with a map $f : Y \to X$ such that 
	\begin{enumerate}
		\item $(Y, X, f)$ is a fiber bundle with smooth manifold fibers.
		\item For every stratum $S \in \Sigma$, $\ker \iota^*_S \omega_n = \ker df|_{S}$
	\end{enumerate}
\end{thm}
Theorem \ref{gotay-intro} essentially says that $\Bbb R^{2n}$ is universal for symplectic reductions in the setup of compact symplectic stratified spaces.
\\

\noindent {\bf Acknowledgments:} We thank Arvind Nair for pointing us to Verona's work \cite{Verona_derham}. It was the comment of Sjamaar-Lerman in  \cite[p. 412]{SL_stratsympred} alluded to above that motivated us to find the equivalence between the algebraic and geometric perspectives in Section \ref{sec-sjamaar}.

\section{Preliminaries on stratified objects}\label{sec-stratprel}

\subsection{Abstractly stratified spaces}\label{sec-ss}
\begin{defn}  Let $X$ be a second countable metric space, and $(I, \leq)$ be a poset. An {\bf $I$-decomposition} of $X$ is a locally finite collection of pairwise disjoint, locally closed subsets $\Sigma = \{S_i : i \in I\}$ of $X$, termed the {\bf strata} of $X$, such that
\begin{enumerate}
\item $X = \bigcup_{i \in I} S_i$
\item $S_i$ is a topological manifold for each $i \in I$
\item $S_i \cap \overline{S_j} \neq \emptyset$ if and only if $S_i \subseteq \overline{S_j}$ if and only if $i \leq j$
\end{enumerate}\end{defn}

We shall carry over the partial order on $I$ along the natural bijection $\Sigma \to I$, $S_i \mapsto i$, to the set of all strata $\Sigma$. This canonically makes the $I$-decomposition on $X$ a $\Sigma$-decomposition. Therefore we shall often omit the indexing set $I$ and simply refer to a {\it decomposition} of $X$, and call the pair $(X, \Sigma)$ a {\it decomposed space}, although we keep the indexing set in the background for notational ease.

Associated to a decomposed space $(X, \Sigma)$, some useful notions are that of {\it depth} and {\it dimension} respectively. Given a stratum $S \in \Sigma$, we define the depth of $S$:
$$\dpt(S) := \sup\{n : \exists \;\text{a chain of strata}\; S = S_0 < S_1 < \cdots < S_n\}$$
The depth of $X$ itself is defined to be $\dpt(X) := \sup\{\dpt(S) : S \in \Sigma\}$. Likewise, the dimension of $X$ is defined to be $\dim(X) := \sup\{\dim(S) : S\in \Sigma\}$. 

For a pair of strata $S, L \in \Sigma$, the condition ``$S \cap \overline{L} \neq \emptyset$ if and only if $S \subset \overline{L}$" indicates that ``deeper" strata always lie in the topological boundary of the ``shallower" strata in $X$. We shall use the notation $\partial L := \overline{L} \setminus L$ for the topological boundary or the frontier of $L$ in $X$. Observe that $\partial L$ is naturally decomposed as it is a union of strata of $X$ that lie deeper than $L$.

We shall use the notion of {\it abstract stratifications} developed by Thom \cite{Thom_stratmaps} and Mather \cite{Mather_notes}. 
 Given a decomposed space $(X, \Sigma)$, a {\it tube system} $\mathscr{T}$ associated to the decomposition is a collection of open sets $\{T_i : i \in I\}$ such that each $T_i$ is an open neighborhood of the stratum $S_i$ in $X$, referred to as the {\it tubular neighborhood} of $S_i$ in $X$, along with a retraction map $\pi_i : T_i \to S_i$ called the {\it tubular projection}, as well as a continuous function $\rho_i : T_i \to [0, \infty)$ such that $\rho_i^{-1}(0) = S_i$. We shall call $\rho_i$ the {\it radial function} associated to the tube.

Given a pair of indices $i \leq j$, one also defines the {\it tubular neighborhood of $S_i$ in $S_j$} to be $T_{ij} := T_i \cap S_j$ and the restrictions of the tubular and radial functions are denoted as $\pi_{ij} := \pi_i|T_{ij}$ and $\rho_{ij} := \rho_i|T_{ij}$. Notice that $\rho_{ij}$ is strictly positive for $i < j$.
We denote a decomposed space $(X, \Sigma)$ equipped with a tube system $\mathscr{T}$ by the  triple $(X, \Sigma, \mathscr{T})$. 

\begin{defn}$(X, \Sigma, \mathscr{T})$ is said to be {\bf abstractly stratified} if
\begin{enumerate}
\item Each stratum $S \in \Sigma$ admit a $\mathcal{C}^\infty$-structure
\item For any pair $i, j \in I$ of indices with $i \leq j$, the tubular projection and radial functions $\pi_{ij}, \rho_{ij}$ are smooth
\item $(\pi_{ij}, \rho_{ij}) : T_{ij} \to S_i \times (0, \infty)$ is a submersion
\item For any triple $i, j, k \in I$ of indices with $i \leq j \leq k$, the following holds:
	\begin{enumerate}
	\item {\bf $\pi$-control condition}: $\pi_{ij} \circ \pi_{jk} = \pi_{ik}$ on $T_{ik} \cap T_{jk} \cap \pi_{jk}^{-1}(T_{ij})$
	\item {\bf $\rho$-control condition}: $\rho_{ij} \circ \pi_{jk} = \rho_{ik}$ on $T_{ik} \cap T_{jk} \cap \pi_{jk}^{-1}(T_{ij})$
	\end{enumerate}
\end{enumerate}
\end{defn}

At first glance this might seem technical, so we elucidate the notion by briefly explaining the purpose of each of these conditions: 

Condition $(1)$ imposes a smooth structure on each stratum, and the rest of the conditions are trying to specify how these smooth structure should interact near the boundary of a stratum with the smooth structures on the deeper strata. 

Conditions $(2)$ and $(3)$ tell us that $(\pi_{ij}, \rho_{ij})$ give \emph{partial cylindrical coordinates} for each of the tubes around the stratum $S_i$ in $S_j$, much like tubular neighborhoods of smooth submanifolds of a manifold. One should imagine $\rho_i$ as a radial distance function from $S_i \subseteq X$. 

Condition $(4)$, or the {\it control conditions}, demand that the tubular neighborhoods of each strata are so arranged that the projection maps are all compatible with each other, and moreover the projection maps  preserve radial co-ordinates. 

\begin{defn}\label{tubes-eq}A pair of abstract stratifications induced by tube systems $\mathscr{T}$ and $\mathscr{T}'$ on a given decomposed space $(X, \Sigma)$ are said to be {\bf equivalent} if $\pi_i, \pi_i'$ and $\rho_i, \rho_i'$ agree on an open neighborhood of $S$ contained in $T_i \cap T_i'$ for all $i \in I$.\end{defn}

\begin{rmk}\label{tubes-germs}We shall often replace a given tube system by a conveniently chosen tube system equivalent to it, and therefore it is best to think of the tubes as defined up to germinal equivalence, and equipped with a germ of a retraction and a germ of a tubular function. \end{rmk}

\begin{rmk}\label{strat-pullback} {\rm Given a pair of abstractly stratified spaces $(X, \Sigma, \mathscr{T})$ and $(X', \Sigma', \mathscr{T}')$, a topological embedding $f : X \hookrightarrow X'$ gives rise to a decomposition of $X$ by $f^* \Sigma := \{f^{-1}(S') : S' \in \Sigma'\}$ and a tube system $f^* \mathscr{T}'$  such that $(X, f^* \Sigma', f^*\mathscr{T}')$ is abstractly stratified as well. More explicitly, elements of 
	$f^* \mathscr{T}'$ are of the form  $T_i' \cap f(X)$, where $T_i' \in 
	\mathscr{T}'$. The tubular projection functions $f^*\pi_i'$ are given by restrictions of $\pi_i'$ to strata of $f(X)$, and radial functions $f^*\rho_i'$ are given by $f^*\rho_i' = \rho_i'\circ f$.}
	
{\rm 	We say $f$ is an {\bf isomorphism} if it is a homeomorphism, $f^* \Sigma' = \Sigma$ and the abstract stratifications are equivalent.}
\end{rmk}

\begin{rmk}\label{strat-pullback2} {\rm 
A similar notion of pullback can be defined if $f : X \rightarrow X'$ is a bundle with  fiber a smooth manifold $F$. Then strata (resp.\ tubular neighborhoods) of $X$ are given by $F-$bundles over strata $S_i'$ (resp.\ induced  $F-$bundles over tubular neighborhoods $T_i'$). Radial functions $f^*\rho_i'$ are given by $f^*\rho_i' = \rho_i'\circ f$ as before. }

 {\rm Next, to define projection functions $f^*\pi_i'$, let $A$ denote the link of $S_i'$ in $X'$. Let $cA$ denote the cone over $A$. Then $T_i'$ is a $cA-$bundle over  $S_i'$. Let $E_i' = f^{-1}(S_i')$ (resp.\  $E(T_i') = f^{-1}(T_i')$) denote the restriction of the bundle $f : X \rightarrow X'$ to $S_i' \subset X'$ (resp.\ $T_i' \subset X'$).
Since $\pi_i': T_i'\to S_i'$ is a strong deformation retract (with contractible fiber $cA$), $E(T_i')$ is  equivalent to $(\pi_i')^* E_i'$. There is a natural bundle map $(\pi_i')^* E_i'\to E_i'$ covering $\pi_i'$. Composing this with the equivalence with $E(T_i')$ furnishes the required  projection function $f^*\pi_i'$.}
\end{rmk}

We follow below the convention of \cite[Section 6]{Mather_notes}:
$\varepsilon > 0$ will denote a positive function on the stratum $S_i$. Let  $$T_i(\varepsilon) := \{ x\in T_i: \ 0 \leq  \rho_i(x) < \varepsilon (\pi_i(x))\}$$ and $T_{ij}(\varepsilon) := T_i(\varepsilon) \cap S_j$ for $i < j$. We shall also denote $$\partial T_i(\varepsilon) :=  \{ x\in T_i: \  \rho_i(x) = \varepsilon(\pi_i(x))\}.$$ Here, we think of $\partial T_i(\varepsilon)$ as the ``unit normal bundle" to the strata $S_i \in \Sigma$. Note that $\partial T_i(\varepsilon)$ admits a natural abstract stratification given by pulling back the stratification $(\Sigma, \mathscr{T})$ on $X$ along the canonical inclusion $\partial T_i(\varepsilon) \hookrightarrow X$. \\

\subsection{Whitney stratified spaces}

The key reason for introducing abstract stratifications is that it happens to be exactly the right intrinsic definition of {\it Whitney stratified sets}, which are a universal model for most well-behaved singular spaces. We briefly recall the preliminaries and introduce the definition. \\

Let $(S, L)$ be a pair of smooth (not necessarily properly) embedded submanifolds of a smooth manifold $M$ such that $S \subset \overline{L}$. 
\begin{enumerate}
\item The pair $(S, L)$ is said to be {\bf $(a)$-regular} if for any sequence $\{x_n\}$ on $L$ converging to a point $y \in S$, such that the sequence of tangent spaces $T_{x_n} L$ converge to a plane $\tau \subset T_y M$, we have $T_y S \subset \tau$. 
\item The pair $(S, L)$ is said to be {\bf $(b)$-regular} if moreover for any sequence $\{y_n\}$ on $S$ which also converges to $y \in S$, such that the secant lines $\overline{x_n y_n}$ converge to a line $\ell \subset T_y M$, we have $\ell \subset \tau$.\\
\end{enumerate}

The notions of convergence of planes and lines mentioned are defined locally, by choosing a coordinate chart $(U, y) \cong (\mathbb R^n, 0)$ around $M$, inside which all but finitely many terms of the sequences $\{x_n\}$ and $\{y_n\}$ belong. The conditions as such are independent of the coordinate charts as can be easily seen using the chain rule, and moreover have invariant coordinate-independent formulations in terms of the Grassmann and secant manifolds (see \cite{Mather_notes} for further discussion). 

\begin{defn}A closed, decomposed subset $(W, \Sigma)$ of a smooth manifold $M$ is said to be {\bf Whitney stratified} if $\Sigma$ consists of smooth (not necessarily properly) embedded submanifolds of $M$, such that any pair of strata $S, L \in \Sigma$ with $S < L$ satisfy $(a)$- and $(b)$-regularity.
\end{defn}

The decomposition of an algebraic variety by its singularities of various depths is not in general a Whitney stratification, a famous counterexample being the Whitney umbrella. However, Whitney proved that a refinement of this stratification does give any algebraic variety the structure of a Whitney stratified set. It was subsequently established by Thom, Lojasiewicz, Hironaka et al that a wide variety of spaces, including semi-algebraic sets and subanalytic sets over either $\mathbb R$ or $\mathbb C$ admit Whitney stratifications.

\begin{thm}\label{strat-whitabs}\cite{Mather_notes} Let $(W, \Sigma)$ be a Whitney stratified subset of a smooth manifold $M$. Associated to each stratum $S \in \Sigma$ there exists a tubular neighborhood $\nu(S) \subset M$ (see discussion following Remark \ref{strat-pullback})
with tubular projection $\pi_S : \nu(S) \to S$, and equipped with a fiberwise norm $\rho_S$ induced from a fiberwise Riemannian metric, such that for all $S, L \in \Sigma$, $\nu(S) \cap \nu(L) \neq \emptyset$ only if $S < L$, in which case
\begin{enumerate}
\item $\pi_S \circ \pi_L = \pi_S$ throughout $\nu(S) \cap \pi_L^{-1}\nu(S)$
\item $\rho_S \circ \pi_L = \rho_L$ throughout $\nu(S) \cap \pi_L^{-1}\nu(S)$
\end{enumerate}
In particular, for any $S \in \Sigma$, define $T_S := \nu(S) \cap W$, equipped with the restrictions $\pi_S|W$ and $\rho_S|W$. This defines a tube system $\mathscr{T}$ on $(W, \Sigma)$ such that $(W, \Sigma, \mathscr{T})$ is abstractly stratified.
\end{thm}

There is a converse to this theorem, which is essentially a Whitney embedding theorem for abstractly stratified spaces. A {\bf realization} of an abstractly stratified space $(X, \Sigma, \mathscr{T})$ in a smooth manifold $M$ is an embedding $\iota : X \to M$ such that $Y := \iota(X) \subset M$ admits the structure of a Whitney stratified subset $(Y, \Sigma')$ of $M$, and is moreover equipped with a tube system $\mathscr{T}'$ as in Theorem \ref{strat-whitabs} so that $(Y, \Sigma', \mathscr{T}')$ is an abstractly stratified space, and the map $\iota : (X, \Sigma, \mathscr{T}) \to (Y, \Sigma', \mathscr{T}')$ is an isomorphism.

\begin{thm}[\cite{teu,goresky-thesis}, see also \cite{Natsume_whitney}]\label{strat-natsume} Any abstractly stratified space $(X, \Sigma, \mathscr{T})$ of finite dimension $n := \dim X$ admits a realization in $\mathbb R^N$ for any $N \geq 2n + 1$.

Furthermore, if $N \geq 2n+2$, any pair of such realizations $\iota_0, \iota_1 : X \to \mathbb R^N$ are isotopic, in the following sense: there is a realization 
$$\mathcal{I} : (X, \Sigma, \mathscr{T}) \times (-\delta, 1 + \delta) \to \mathbb R^N$$ 
such that $\mathcal{I}(-, t)$ is a realization for all $t \in [0, 1]$ and $\mathcal{I}(-, 0) = \iota_0$, $\mathcal{I}(-, 1) = \iota_1$.
\end{thm}

It was conjectured by Whitney that any complex analytic variety admits a stratification such that the  neighborhood of any point in a stratum is fibered by links of the stratum. This conjecture was formulated and proved by Thom and Mather using the theory of controlled vector fields, see \cite{Mather_notes}.

Before stating the theorem, we introduce the notion of a {\it cone} on a stratified space. Let $(X, \Sigma, \mathscr{T})$ be an abstractly stratified space. Then we define the open cone on $X$ to be the topological space $cX := X\times [0, 1)/X \times \{0\}$, and we denote $v \in cX$ to be the vertex of the cone, given by the image of $X \times \{0\}$. $cX$ admits a natural stratification given by $c\Sigma := \Sigma \times (0, 1) \sqcup \{v\}$ and a tube system $c\mathscr{T} := \{T_i \times (0, 1), \pi_i \times \mathrm{id}, \rho_i, i \in I\} \cup \{B_v, \pi_v, \rho_v\}$ where $B_v \subset X$ is a ball around $v$, $\pi_v : B_v \to \{v\}$ is the constant map, and $\rho_v : B_v \to [0, \infty)$ is, up to reparametrizations on the codomain, the natural height function $cX \to [0, 1)$ on the cone.

\begin{thm}\label{thom-isotopy} \cite[p. 41]{GM_SMT}
	Let $(X, \Sigma, \mathscr{T})$ be an abstractly stratified space. Let $x \in X$ be a point and $S \in \Sigma$ be the unique stratum of $X$ containing $x$ of dimension $n := \dim S$.

There exists an abstractly stratified space $(L, \Sigma_\ell, \mathscr{T}_\ell)$ called the {\bf link} of $S$ such that for a sufficiently small neighborhood $U \subset X$ of $x$ equipped with the induced stratification from the canonical inclusion, $U$ is isomorphic to $\mathbb R^n \times cL$.\end{thm}

\subsection{Compression lemma}\label{sec-compression}

In this subsection  we describe a technical construction that we shall often use to enforce control on uncontrolled geometric data on abstractly stratified spaces.
Consider the following basic motivating example. Let $(M, \partial M)$ be a manifold with boundary and $T := \partial M \times [0, \infty)$ be a collar neighborhood of $\partial M$ inside $M$, with collar projection $\pi : T \to \partial M$. Consider a smooth function $\psi : [0, \infty) \to [0, 1]$ such that $\psi(t) = 0$ if $t \leq 1/4$, $\psi(t) = 1$ if $t \geq 3/4$, and $\psi$ is strictly increasing on $[1/2, 3/4]$. Then the map $\Pi : (M, \partial M) \to (M, \partial M)$, defined by $\Pi \equiv 1$ on $M \setminus T(3/4)$ and $\Pi(x, t) = (x, \psi(t) t)$ on $T(3/4)$ ``compresses" the smaller collar $T(1/4)$ to the boundary $\partial M$ while leaving everything outside a larger collar $T(3/4)$ fixed. Observe $\Pi|T(1/4) \equiv \pi$ is the collar projection. \\

We shall generalize the above construction to any abstractly stratified space $(X, \Sigma, \mathscr{T})$ to obtain a {\it compression map} $\Pi : X \to X$ which will ``compress" tubular neighborhoods of each stratum down to the stratum itself while leaving everything outside a larger tubular neighborhood fixed pointwise. Moreover, this shall appear as the time-1 map of a homotopy $\mathcal{H} : X \times I \to X$ which progressively shrinks these tubular neighborhoods to the strata. To carry this out, we shall require the {\it family of lines} machinery developed by Goresky \cite{Goresky_lines}.

\begin{thm}\label{strat-fol}\cite{Goresky_lines} Let $(X, \Sigma, \mathscr{T})$ be an abstractly stratified space. Then there exists $\delta > 0$ and a family of retractions $r_i(\varepsilon) : T_i \setminus S_i \to \partial T_i(\varepsilon)$ for every $0 < \varepsilon < \delta$ which satisfy the following  relations:
\begin{enumerate}
\item $r_i(\varepsilon') \circ r_j(\varepsilon) = r_j(\varepsilon') \circ r_i(\varepsilon)$ for all $i \leq j \in I$, $0 < \varepsilon' < \delta$.
\item  $\rho_i \circ r_j(\varepsilon) = \rho_i$ and $\rho_j \circ r_i(\varepsilon) = \rho_j$ for all $i \leq j$, $i, j \in I$.
\item $\pi_i \circ r_j(\varepsilon) = \pi_i$ for all $i \leq j$, $i, j \in I$.
\item $r_i(\varepsilon') \circ r_i(\varepsilon) = r_i(\varepsilon')$ for all $i \in I$, $0 < \varepsilon < \varepsilon' < \delta$.
\item $r_i(\varepsilon)|T_i(\varepsilon) \cap S_j : T_i(\varepsilon) \cap S_j \to \partial T_i(\varepsilon) \cap S_j$ is smooth for all $i \leq j$, $i, j \in I$.
\end{enumerate}
\end{thm}

Equipped with such a family of retractions, we can define a stratum-preserving homeomorphism $h_i : T_i \setminus S_i \to \partial T_i(\varepsilon) \times (0, \infty)$ given by $h_i(p) = (r_i(\varepsilon)(p), \rho_i(p))$. This homeomorphism extends to a stratum-preserving homeomorphism $h_i : T_i \to M(\pi_i)$. Here $M(\pi_i) := (\partial T_i(\varepsilon) \times [0, \infty) \sqcup S_i)/(x, 0) \sim \pi_i(x)$ is the semi-infinite mapping cylinder of $\pi_i|S_i(\varepsilon) : \partial T_i(\varepsilon) \to S_i$. \\

Let $(X, \Sigma, \mathscr{T})$ be a finite-dimensional abstractly stratified space. Choose a smooth function $\psi : [0, \infty) \to [0, 1]$ as before such that $\psi(x) = 0$ for all $x \leq \varepsilon$, $\psi(x) = 1$ for all $x \geq 3\varepsilon$ and $\psi$ is strictly increasing on $[\varepsilon, 3\varepsilon]$. Define
\begin{gather*}\phi_i : \partial T_i(\varepsilon) \times [0, \infty) \times [0, 1] \to \partial T_i(\varepsilon) \times [0, \infty) \\ 
\phi_i((x, t), s) := (x, s \psi(t)t + (1- s) t)\end{gather*}
$\phi_i$ simply performs the compression trick described earlier for a manifold-with-boundary but towards the end $\partial T_i(\varepsilon) \times \{0\}$ of the cylinder, and we keep track of the compression procedure through time, measured by the $s$-coordinate; at time $s = 0$ the compression process is performed to its completion. Note that this descends to a homotopy $\phi_i : M(\pi_i) \times I \to M(\pi_i)$ on the mapping cylinders. This is the map which we shall use to compress near a single stratum $S_i \in \Sigma$. 

Let $\mathcal{H}_i : X \times I \to X$ be defined, for all $s \in I$, by
$$\mathcal{H}_i(-, s) = \begin{cases}\mathrm{id} & \text{throughout}\; X \setminus T_i \\  h_i^{-1} \circ \phi_i(-, s) \circ h_i & \text{in}\; T_i(\varepsilon)\end{cases}$$ 

The collection of these ``local compression" homotopies $\{\mathcal{H}_i(-, s) : i \in I_0\}$ commute for each time $s \in I$ due to the choice of our family of retractions from Theorem \ref{strat-fol}. Therefore, we can compose these maps in \textbf{any order} to obtain a map $\mathcal{H} : X \times I \to X$ which performs all of these local compressions near every strata at once. This produces the desired compression homotopy $\mathcal{H} : X \times I \to X$.

Observe that $\phi_i(-, 1) = \pi_i$ on $\partial T_i(\varepsilon) \times [0, \varepsilon]$, so $\mathcal{H}_i(-, 1) = \pi_i$ on $T_i(\varepsilon)$ and $H_i(-, 1) \equiv \mathrm{id}$ on $X \setminus T_i(\varepsilon)$ as before. The final map $\Pi := \mathcal{H}(-, 1)$ is a composition of the family $\{\mathcal{H}_i(-, 1) : i \in I\}$ in any order.

To clarify we exhibit the behavior of a factor of $\Pi$ given by the composition of these maps in increasing order of depth along a maximal chain in $\Sigma$. Let $S_1 < S_2 < \cdots < S_n$ be such a chain. Let us denote 
$$\varphi_{p,q} := \mathcal{H}_{p}(-, 1) \circ \cdots \circ \mathcal{H}_{q}(-, 1), \;\;\; 1 \leq p < q \leq n$$
Fix any $x \in X$, let $d = \min\{k : x \in T_{k}(\varepsilon)\}$ be defined so that $S_d$ is the deepest stratum to which $x$ is $\varepsilon$-close. We claim $\varphi_{1, n}(x) = \varphi_{1, d-1}(\pi_{d, n}(x))$. This follows from induction using the observation that if $x \in T_{k, n}(\varepsilon)$ and $k < \ell$, $\rho_{k, \ell}(\pi_{\ell, n}(x)) = \rho_{k, n}(x) \leq \varepsilon$, implying $\pi_{\ell, n}(x) \in T_{k, \ell}(\varepsilon)$.
This establishes that for any $x \in X$, if $D \in \Sigma$ denotes the deepest stratum of $X$ such that $x$ is $\varepsilon$-close to $D$, $\Pi(x) = \Pi(\pi_D(x))$, since $\Pi|\overline{D}$ is itself a compression map on $(\overline{D}, \Sigma \cap \overline{D}, \mathscr{T} \cap \overline{D})$ obtained from composition of the maps in the family $\{\mathcal{H}_i(-, s) : i \in \Sigma \cap \overline{D}\}$. 

Finally, we discuss smoothness properties of $\Pi$. Let $S<L$ and let $A$ denote the link of $S$ in $X$. Let $A_L = A \cap L$. Then the tubular neighborhood of $S$ in $L$ is of 
the form $U \times cA_L$, where $cA_L$ denotes the cone on $A_L$, and $U\subset S$ is an open Euclidean chart. 
This follows from Thom's first isotopy lemma (see \cite[Section 2]{ms-hprin}, \cite[Lemma 2.15]{ms-hprin} for further details). 
Consider local coordinates on $U \times cA_L$ of the form
$\R^s \times \R^l \times [0,1]$, where $\R^s \subset S$ is an open chart in $S$,  $\R^l \subset A_L$ is an open chart in $L$, and $[0,1]$ is the radial (cone) direction in $cA$. We can choose $\Pi$ so that it can be given in local charts as 
$$\Pi (x,y,t) = (x,y, \psi(t)),$$
where $\psi :[0,1] \to [0,1]$ is a smooth function satisfying the following:
\begin{enumerate}
\item in local chart coordinates, $0$ is sent to the cone-point $c$,
\item there exists $0<c<1$ such that $\psi ([0,c]) =\{0\}$,
\item $\psi$ maps $(c,1)$ diffeomorphically to $(0,1)$.
\end{enumerate}
If $\Pi$ satisfies the above properties locally, we say henceforth that it is
\emph{smooth}. This implies, in particular, that if $(X, \Sigma, \mathscr{T})$ is an abstractly stratified space where the underlying space $X$ is a \emph{smooth manifold}, then $\Pi$ is a smooth map. 

We summarize the properties of the construction above in the following proposition.

\begin{prop}\label{strat-compress}For an abstractly stratified space $(X, \Sigma, \mathscr{T})$, there exists a map $\Pi : X \to X$, a homotopy $\mathcal{H} : X \times [0, 1] \to X$ and $\varepsilon > 0$, such that
\begin{enumerate}
\item $\mathcal{H}(-, 1) = \Pi$
\item $\mathcal{H}(-, 0) = \mathrm{id}$
\item $\mathcal{H}(-, s) : X \to X$ is an isomorphism for all $s \in [0, 1)$
\item For any pair of strata $S, L \in\Sigma$ with $S < L$, $\Pi \equiv \mathrm{id}$ throughout $L \setminus T_{S, L}(3\varepsilon)$
\item For any $x \in X$, $\Pi(x) = \Pi(\pi_D(x))$ where $D$ is the deepest stratum such that $T_D(\varepsilon)$ contains $x$.
\item $\Pi$ is smooth. 
\end{enumerate}
\end{prop}

\begin{cor}\label{cor-cprsctrl}For any stratum $S \in \Sigma$, $\Pi \circ \pi_S = \Pi$ on $T_S(\varepsilon)$.\end{cor}

\begin{proof}Indeed, let $x \in T_S(\varepsilon)$, then $\Pi_1(x) = \pi_D(x)$ where $D < S$ is the deepest stratum such that $x \in T_D(\varepsilon)$. Note that $\rho_D(\pi_S(x)) = \rho_D(x) \leq \varepsilon$, therefore $\pi_S(x) \in T_D(\varepsilon)$ as well. If there is a deeper stratum $D' < S$ such that $\pi_S(x) \in T_{D'}(\varepsilon)$, then $\rho_{D'}(x) = \rho_{D'}(\pi_S(x)) \leq \varepsilon$, therefore $x \in T_{D'}(\varepsilon)$ for a stratum $D' < S$ with $\dpt(D') < \dpt(D)$, contradicting the choice of $D$. Therefore, 
$$(\Pi \circ \pi_D)(x) = \Pi(\pi_D(x)) = \pi_D(\pi_S(x)) = \pi_D(x) = \Pi(x)$$ 
as required.\end{proof}

\begin{rmk}\label{rmk-compskel} Note that the map $\Pi$ stands out as unusual in the theory of stratified spaces so far since it does not preserve strata; indeed, it pushes parts of each stratum down to its boundary. We remark however that $\Pi$ is {\bf skeleton-preserving} in the sense that $\Pi(\overline{S}) \subset \overline{S}$ for any stratum $S \subset \Sigma$. Moreover, $\Pi|_{\overline{S}}$ itself is a compression map, therefore we can effectively understand $\Pi$ by induction, using compression maps on stratified spaces of lower depth than that of $X$. \end{rmk}

\section{de Rham cohomology on stratified spaces}\label{sec-stratdr}

Given any abstractly stratified space $(X, \Sigma, \mathscr{T})$, the ``coarsest" notion of a differential $k$-form on $X$ is a collection of $k$-forms $\{\omega_S : S \in \Sigma\}$ defined on each stratum. This defines a complex $\bigoplus_{S \in \Sigma} \Omega^\bullet(S)$ given by the direct sum of the de Rham complexes of each stratum. However, the cochain groups are too large for our purposes, as we have so far not demanded any condition on how the forms on various strata should fit together; the complex is independent of the chosen tube system $\mathscr{T}$. For example, the cohomology groups are $\bigoplus_{S \in \Sigma} H^\bullet(S; \Bbb R)$. This does not reflect the topology of $X$.
We propose two ways to remedy this:

\begin{enumerate}

\item By equipping $X$ with a realization $X \hookrightarrow M$ and using the pullback of the sheaf of smooth differential forms on $M$ along this embedding. This gives us a way of saying that a collection of forms $\{\omega_S : S \in \Sigma\}$  ``fit together", i.e.\ there is a germ of a smooth form $\widetilde{\omega}$ along the image of $X$ in $M$ such that $\widetilde{\omega}|_S = \omega_S$ for every stratum $S \in \Sigma$.

\item By defining an appropriate subcomplex of $\bigoplus_{S \in \Sigma} \Omega^\bullet(S)$ which consists of collections $\{\omega_S : S \in \Sigma\}$ such that whenever $S < L$, $\omega_L$ is ``radially compactly supported" while approaching $S$, or more precisely, $\pi_{S, L}^* \omega_S = \omega_L$. This notion of ``fitting together" is more restrictive than the first notion, but is completely intrinsic to the stratified space $X$. This approach was taken by Verona in \cite{Verona_derham}.

\end{enumerate}

We discuss and compare these two notions in this section.

\subsection{Smooth de Rham cohomology}

Let $(W, \Sigma) \subset M$ be a Whitney stratified subset of a smooth manifold $M$. The smooth structure on $M$ gives a natural way to speak of a smooth structure on $W$. To be precise, let $\mathcal{C}^\infty_M$ be the sheaf of smooth functions on $M$. We can restrict this sheaf to the subspace $\iota : W \hookrightarrow M$ to obtain the {\it sheaf of smooth functions} on $W$, which we shall denote as $\mathcal{C}^\infty_W := \iota^* \mathcal{C}^\infty_M$. This can be concisely described as the sheaf of germs of smooth functions defined along $W$. Note that $\mathcal{C}^\infty_M$ is fine by the existence of bump functions on smooth manifolds, and the fact that restriction of fine sheaves over metric space to closed subspaces remains fine. Hence $\mathcal{C}^\infty_W$ is a fine sheaf on $W$. We shall drop the subscript when the ambient stratified space is understood and simply denote the sheaf as $\mathcal{C}^\infty$.

Let $\Omega^\bullet_M$ denote the sheaf of differential forms on $M$. We can likewise define $\Omega^p_W := \iota^*\Omega^p_M$ to be the {\it sheaf of smooth differential $p$-forms} on $W$, and these are likewise fine sheaves on $W$. We once again drop the subscript whenever the ambient stratified space is understood and denote the sheaf simply as $\Omega^\bullet$. The exterior derivative operator $d : \Omega^p_M \to \Omega^{p+1}_M$ restricts to an exterior derivative operator on  $\Omega^\bullet$. In particular, it gives rise to a cochain complex $(\Omega^\bullet(W), d)$ at the level of global sections. We shall refer to this as the {\it smooth de Rham complex} for $W$. The cohomology of the complex shall be referred to as the {\it smooth de Rham cohomology} of $W$, $H^\bullet_{dR}(W) := H^\bullet(\Omega^\bullet(W), d)$.

\begin{rmk}\label{rmk-standingnbhd}
Here, $\Omega^p_W$ is  a sheaf of smooth forms obtained by restricting smooth p-forms
from a small neighborhood in $M$. 
	Thus $\Omega^\bullet$ is really the space of smooth forms in a small tubular neighborhood. Since the stratified space is a deformation retract, the cohomologies are preserved by homotopy-invariance. The de Rham cohomology theory we thus obtain is independent of the embedding by Verona's Theorem \ref{thm-veronadr} below. Thus,  for the purposes of this paper, de Rham cohomology is realized by an actual embedding and taking a {\bf fixed} sufficiently
small neighborhood in a smooth manifold. We shall implicitly assume such an embedding throughout.
\end{rmk}

Recall that the de Rham complex of sheaves,
$$0 \to \underline{\mathbb R} \to \Omega^1_M \stackrel{d}{\to} \Omega^2_M \stackrel{d}{\to} \cdots$$
defines a resolution of the constant $\mathbb R$-valued sheaf on $M$. Therefore, so does
$$0 \to \underline{\mathbb R} \to \Omega^1 \stackrel{d}{\to} \Omega^2 \stackrel{d}{\to} \cdots$$
Indeed, it suffices to check this at the level of stalks, and the restriction sheaves on $W$ are indistinguishable from the corresponding sheaves on $M$ at the level of stalks. Moreover, from our comments above this is a resolution by fine --- thus acyclic --- sheaves. Therefore, the sheaf cohomology of $\underline{\mathbb R}$ is isomorphic to the de Rham cohomology of $W$. Moreover, the sheaf cohomology of $\underline{\mathbb R}$ is isomorphic to the singular cohomology of $W$ valued in $\mathbb R$, as $W$ is locally contractible. As a consequence, 
$$H^\bullet_{dR}(W) \cong H^\bullet_{sing}(W; \mathbb R)$$

Next, we would like to use the cohomology theory defined above for abstractly stratified spaces. To do so, we introduce the following notion: 

\begin{defn}\label{strat-smooth} A {\bf smooth structure} on an abstractly stratified space is a choice of a realization (see Theorem \ref{strat-natsume}) in a smooth manifold. \end{defn}

The rationale behind this definition is as follows. Suppose $(X, \Sigma, \mathscr{T})$ is an abstractly stratified space, equipped with a smooth structure, i.e., a realization $j : X \to M$ in some smooth manifold $M$. Then the image $W := \iota(X) \subset M$ is a Whitney stratified set, and therefore admits the sheaf of smooth functions $\mathcal{C}^\infty_W$ defined earlier; we can pullback the sheaf to obtain a sheaf $\mathcal{C}^\infty_X := j^* \mathcal{C}^\infty_W$ on $X$, which we call the {\it sheaf of smooth functions} on $X$. This notion depends on the choice of the realization in a crucial way, as the smooth isotopy extension theorem fails for isotopies of embeddings of stratified spaces. 

We shall similarly define $\Omega_X^\bullet := j^* \Omega_W^\bullet$ as the {\it sheaf of smooth differential forms} on $X$, and drop the subscript from both notations whenever the ambient stratified space along with the smooth structure $(X, j)$ is understood. The discussion above yields a fine resolution $0 \to \underline{\mathbb R} \to \Omega^\bullet$, and hence a cochain complex $(\Omega^\bullet(X), d)$ and a cohomology theory $H^\bullet_{dR}(X)$ isomorphic to the real-valued singular cohomology of $X$, for the abstractly stratified space $X$.

\subsection{Controlled de Rham cohomology}

Let $(X, \Sigma, \mathscr{T})$ be an abstractly stratified space. A continuous function $f : X \to \mathbb R$ will be said to be \textit{controlled} if for every stratum $S \in \Sigma$, $f|S$ is a smooth function, and for any pair of indices $i, j \in I$, $i < j$, $f \circ \pi_{ij} = f$ throughout $T_{ij}$. A useful way of thinking of this condition is by the slogan ``controlled functions are constant on links".

For any open subset $U \subseteq X$,  we obtain an induced abstract stratification on $U$ by pulling back the abstract stratification on $X$ along the inclusion $U \hookrightarrow X$ (cf.\ Remark \ref{strat-pullback}). Let $\mathcal{C}^\infty_{co}(U)$ denote the $\mathbb R$-algebra of controlled functions on $U$. It is clear that the association $U \mapsto \mathcal{C}^\infty_{co}(U)$ defines a sheaf of $\mathbb R$-algebras on $X$; we denote this sheaf by $\mathcal{C}^\infty_{co}$ and call it the {\it sheaf of controlled functions}.

\begin{thm}$\mathcal{C}^\infty_{co}$ is a fine sheaf on $X$, i.e., admits a partition of unity.\end{thm}

\begin{proof}It suffices to construct, for any given point $x \in X$ and any neighborhood $U \subset X$ of $x$, a controlled function $f : X \to \mathbb R$ such that $0 \leq f \leq 1$, $\supp(f) \subset U$ and $f \equiv 1$ throughout a smaller open neighborhood $V \subset X$ of $x$. We use induction on depth to accomplish this. 

Note that an abstractly stratified space of depth $0$ is a manifold, in which case this follows from the standard construction of bump functions. 

By induction, let us assume that theorem is true for all abstractly stratified spaces of depth less than $\dpt(X)$. Let $S \in \Sigma$ denote the unique stratum of $X$ that contains $x$. Let $K \subset S$ be a compact subset containing $x$ and choose $g \in \mathcal{C}^\infty(S)$ such that $0 \leq g \leq 1$, $\supp(g) \subseteq K$ and $g \equiv 1$ throughout an open neighborhood of $x$ contained in $K$. Let $T$ denote the tubular neighborhood of $S$, and $\pi, \rho$ be the associated projection and radial function. 

Define $\tilde{g} \in \mathcal{C}^\infty_{co}(T)$ by $\tilde{g} := g \circ \pi$. Note $\dpt(T \setminus S) < \dpt(X)$. Therefore, by induction hypothesis, there exists $h \in \mathcal{C}^\infty_{co}(T \setminus S)$ with $0 \leq h \leq 1$, $\supp(h) \subset T(\varepsilon)$ and $h \equiv 1$ throughout $T(\varepsilon/2) \setminus S$. Define $\tilde{h} \in \mathcal{C}^\infty_{co}(T)$ by  $\tilde{h}|(T \setminus S) = h$ and $\tilde{h}|S = 1$. We choose $\varepsilon > 0$ and $K \subset S$ such that $\pi^{-1}(K) \cap T(\varepsilon) \subset V$, and define $f \in \mathcal{C}^\infty_{co}(X)$ by $f|T = \tilde{g} \tilde{h}$ and $f|(X \setminus T) = 0$. This is our desired function. \end{proof}

In \cite{Verona_derham}, Verona introduced a controlled notion of differential forms on abstractly stratified spaces. We recall the definition:

\begin{defn}A {\bf controlled differential $p$-form} on an abstractly stratified space $(X, \Sigma, \mathscr{T})$ is a collection  $\omega := \{\omega_i : i \in I\}$ such that for each $i \in I$, $\omega_i$ is a $p$-form on the stratum $S_i\in \Sigma$, such that whenever $i, j \in I$ is a pair of indices, $i < j$, the following control condition holds:
$$\pi_{ij}^* \omega_j = \omega_i \; \text{on} \; T_{ij}$$ \end{defn}

As before, we define a sheaf $\Omega_{co}^p$ of controlled differential $p$-forms on $X$. Note that controlled $0$-forms are simply controlled functions, thus $\Omega_{co}^0 = \mathcal{C}^\infty_{co}$. For each open set $U \subset X$, $\Omega_{co}^p(U)$ is naturally a $\mathcal{C}^\infty_{co}(U)$-module, hence $\Omega_{co}^p$ is a sheaf of $\mathcal{C}^\infty_{co}$-modules. Therefore, $\Omega_{co}^p$ is also a fine sheaf.\\

On a manifold $M$, a $p$-form can be written in a neighborhood of any point $x \in M$ as a linear combination of {\it primitive $p$-forms} of the form $f dg_1 \wedge \cdots \wedge dg_p$ where $f, g_1, \cdots, g_p \in \mathcal{C}^\infty_x$. We give an analogous local description for controlled $p$-forms on an abstractly stratified space $(X, \Sigma, \mathscr{T})$. 

Suppose $x \in X$ is a point, then there is a unique stratum $S \in \Sigma$ containing $x$, and any neighborhood of $x$ in $X$ must contain a further neighborhood of the form $(\pi_S)^{-1}(B) \cap T_S(\varepsilon) = \bigsqcup_{L \geq S} (\pi_{S, L})^{-1}(B) \cap T_{S, L}(\varepsilon)$ where $B \subset S$ is a neighborhood of $x$ in $S$.

By the last paragraph, $\omega_S$ is a linear combination of primitive forms $f dg_1 \wedge \cdots \wedge dg_p$ where $f, g_1, \cdots, g_p \in \mathcal{C}^\infty(B)$ on the ball $B \subset S$, after shrinking the ball if necessary. The functions $f, g_1, \cdots, g_p$ have a unique controlled extension on $(\pi_S)^{-1}(B)$ by setting $\tilde{f} := f \circ \pi_{S, L}$ and $\tilde{g}_i := g_i \circ \pi_{S, L}$ on $\pi_{S, L}^{-1}(B)$, for $1 \leq i \leq p$, for every stratum $L \in \Sigma$ containing $S$ as a boundary stratum. Therefore, we obtain the controlled $p$-form $\tilde{f}_i d\tilde{g}_1 \wedge \cdots \wedge d\tilde{g}_p$ defined in a neighborhood of $x$ in $X$. 

For every stratum $L \geq S$ of $X$, $\omega_L = (\pi_{S, L})^* \omega_S$ which is a linear combination of $(\pi_{S, L})^*(f dg_1 \wedge \cdots \wedge dg_p) = \tilde{f} d\tilde{g}_1 \wedge \cdots \wedge d\tilde{g}_p$ in a neighborhood of $x$ in $X$. Therefore,  $\Omega_{co, x}^p$ is generated by germs of {\it controlled primitive $p$-forms} $f dg_1 \wedge \cdots \wedge dg_p$ where $f, g_1, \cdots, g_p \in \mathcal{C}^\infty_{co, x}$ are germs of controlled function around $x$. 

Notice that if $\omega = \{\omega_S : S \in \Sigma\}$ is a controlled $p$-form, $d\omega := \{d\omega_S : S \in \Sigma\}$ is a controlled $(p+1)$-form, as $\pi_i^*$ commutes with $d$. Therefore, the exterior derivative gives a sheaf homomorphism $d : \Omega_{co}^p \to \Omega_{co}^{p+1}$. 

\begin{thm}[Controlled Poincar\'e lemma]\label{ctrl-lem-poincare} The complex of sheaves 
$$\Omega_{co}^{p-1} \stackrel{d}{\to} \Omega_{co}^p \stackrel{d}{\to} \Omega_{co}^{p+1}$$
is exact.\end{thm}

\begin{proof}It suffices to prove exactness at the level of stalks. That is, suppose $\alpha \in \Omega_{co, x}^p$ is the germ of a closed controlled $p$-form around a neighborhood of $0$, i.e, $\alpha \in \Omega_{co}^p(U)$ for some neighborhood $U \subset X$ of $x$ such that $d\alpha = 0$. Let $S \in \Sigma$ be the unique stratum of $X$ containing $U$. Then, in particular, $\alpha_S$ defines a closed form in the neighborhood $S \cap U$ of $x$ in $S$. Therefore, by the Poincar\'e lemma for manifolds, there exists a $(p-1)$-form $\beta_S$ on a ball neighborhood $B \subset S \cap U$ of $x$ such that $\alpha_S = d\beta_S$.

Once again, we provide a controlled extension of $\beta_S$ by defining $\beta_L = \pi_{S, L}^* \beta_S$ on $\pi_{S, L}^{-1}(B)$ for all $L \in \Sigma$, $L \geq S$. Then $\beta := \{\beta_L : L \in \Sigma\}$ defines a germ of a controlled $(p-1)$-form around a neighborhood of $x$ such that $d\beta = \alpha$, since $d\beta_S = \alpha_S$ and $d$ commutes with $\pi_{S, L}^*$. This concludes the proof.\end{proof}

Let $\Omega_{co}^p(X)$ be the $\mathbb R$-vector space of controlled differential $p$-forms on $X$. These define a {\it controlled de Rham complex} $(\Omega^\bullet_{co}(X), d)$. 

We obtain from Theorem \ref{ctrl-lem-poincare} that the complex of sheaves 
$$0 \to \underline{\mathbb R} \stackrel{d}{\to} \Omega_{co}^0 \stackrel{d}{\to} \Omega_{co}^1 \stackrel{d}{\to} \cdots$$
is exact. Moreover, each $\Omega_{co}^p$ is a fine sheaf and therefore acyclic. Hence the above is a resolution of the constant sheaf $\underline{\mathbb R}$. Therefore, the sheaf cohomology of $\underline{\mathbb R}$ computes the cohomology of the controlled de Rham complex $(\Omega^\bullet_{co}(X), d)$. On the other hand, the sheaf cohomology of $\underline{\mathbb R}$ is the $\mathbb R$-valued \v{C}ech cohomology of $X$. By Theorem \ref{thom-isotopy}, $X$ is locally contractible, therefore sheaf cohomology is isomorphic to singular cohomology. Combining these, we obtain:

\begin{thm}[Verona's de Rham theorem]\label{thm-veronadr} For any abstractly stratified space $(X, \Sigma, \mathscr{T})$, the cohomology of the complex of controlled differential forms computes the singular cohomology of $X$ with real coefficients:
$$H^\bullet(\Omega^\bullet_{co}(X), d) \cong H^\bullet_{sing}(X; \mathbb R)$$\end{thm}

We shall henceforth call the cohomology of the complex of controlled differential forms as the {\it controlled de Rham cohomology} of the abstractly stratified space $(X, \Sigma, \mathscr{T})$ and denote it as $H^\bullet_{co}(X) := H^\bullet(\Omega^\bullet_{co}(X), d)$.

\subsection{Isomorphism of cohomology theories}

We have thus obtained three abstractly isomorphic cohomology theories: $H^\bullet_{dR}, H^\bullet_{co}$ and $H^\bullet_{sing}$. In this section we give an explicit isomorphism between the three theories. 

If $(X, \Sigma, \mathscr{T})$ is an abstractly stratified space, we can always replace the tube system $\mathscr{T}$ by an equivalent tube system $\mathscr{T}'$ and consider a homotopy $\mathcal{H} : X\times I \to X$ compressing the new tube system $\mathscr{T}'$ using Proposition \ref{strat-compress} on the equivalent abstract stratification $(X, \Sigma, \mathscr{T}')$ instead. Often we shall need to choose $\mathscr{T}'$ such that the tubes are of sufficiently small radius, cf. Remark \ref{tubes-germs}. In such cases, we shall indicate such a choice by declaring the compression map  {\it sufficiently thin}.

Let $(X, \Sigma, \mathscr{T})$ be an abstractly stratified space equipped with a smooth structure $j : X \to M$. By a slight abuse of notation we identify $X$ with $j(X)$ and treat it as a subset of $M$. We equip $M$ with an abstract stratification as follows: let $\Sigma' := \Sigma \cup \{M \setminus X\}$ and $\mathscr{T}'$ consist of the tubular neighborhood system on the strata of $\Sigma$ as in Theorem \ref{strat-whitabs}. We declare the tube around the open top dimensional stratum $M \setminus X$ to be itself. This defines a valid abstract stratification $(M, \Sigma', \mathscr{T}')$ on $M$ by Theorem \ref{strat-natsume} and Theorem \ref{strat-whitabs}.

Let $\Pi := \mathcal{H}(-, 1) : M \to M$ be a sufficiently thin compression map for the abstractly stratified space $(M, \Sigma', \mathscr{T}')$. We choose $U := \bigcup_{S \in \Sigma} \nu(S)$ to be a union of tubular neighborhoods of  strata as in Theorem \ref{strat-whitabs} satisfying $\Pi(U) = X$. We shall think of $U$ as an analog of the manifold normal neighborhood for the stratified subset $X \subset M$, and $\Pi : U \to X$ as an analog of the normal projection. 

First, we show that a controlled differential form on $X$ has a natural smooth extension to germ of a smooth differential form along $X \subset M$. This will give a canonical embedding of subcomplexes $\Omega^\bullet_{co}(X) \subset \Omega^\bullet(X)$.

\begin{prop} Let $\omega = \{\omega_S : S \in \Sigma\}$ be a controlled differential form on $X$. Then the collection of forms $\{\pi_S^*\omega_S \in \Omega^\bullet(\nu(S)) : S \in \Sigma\}$ patch to give a germ of a smooth differential form $\widetilde{\omega}$ on $U$.\end{prop}

\begin{proof}Let $S, L \in \Sigma$ be a pair of strata, with $S < L$. Then,
$$\pi_L^*\omega_L = \pi_L^*\pi_{S, L}^*\omega_S = (\pi_{S, L} \circ \pi_L)^* \omega_S = \pi_S^*\omega_S$$
throughout $\nu(S) \cap \pi_L^{-1}(T_{S, L}) = \nu(S) \cap \pi_L^{-1}\nu(S)$. Now, choose a smaller tubular neighborhood $\nu'(S) \subset \nu(S)$ if necessary such that $\nu'(S) \cap \nu(L) \subset \pi_L^{-1}(\nu(S))$. Then $\pi_L^*\omega_L = \pi_S^*\omega_S$ on $\nu'(S) \cap \nu(L)$. 

Therefore, by adjusting tubular neighborhoods downward along strata if necessary, we can choose a collection of tubular neighborhoods $\{\nu(S) : S \in \Sigma\}$ such that $\pi_S^*\omega_S \in \Omega^\bullet(\nu(S))$ patch to a well-defined smooth differential form $\widetilde{\omega}$ on the open neighborhood $\bigcup_{S \in \Sigma} \nu(S)$ of $X \subset M$, as required.\end{proof}

Let $\eta = \{\eta_S : S \in \Sigma\}$ be a smooth differential form on $X$. We choose an extension $\widetilde{\eta}$ to a smooth differential form in an open neighborhood $U$ of $X \subset M$. $U$ inherits an induced stratification from $(M, \Sigma', \mathscr{T}')$ as usual; we treat $\widetilde{\eta} = \{\eta_S : S \in \Sigma'\}$ as collection of differential forms on strata of $U$, where $\eta_{U \setminus X} := \widetilde{\eta}|(U \setminus X)$.

\begin{prop}\label{form-compress}  $\Pi^*\widetilde{\eta} = \{\iota_S^* \Pi^*\widetilde{\eta} : S \in \Sigma'\}$ is a controlled differential form on $U$.\end{prop}

\begin{proof}Let $S, L \in U \cap \Sigma'$ be a pair of strata, with $S < L$. Let $\iota_S, \iota_L : S, L \to U$ denote the canonical inclusion maps and $\pi_S : \nu(S) \to S \subset U$ be the tubular projection. Then, as $\pi_{S, L} = \pi_S \circ \iota_L$, we have:
$$\pi_{S, L}^* (\iota_S^*\Pi^* \widetilde{\eta}) = \iota_L^* \pi_S^* \iota_S^* \Pi^* \widetilde{\eta} = \iota_L^* (\Pi \circ \iota_S \circ \pi_S)^* \widetilde{\eta} = \iota_L^*(\Pi \circ \pi_S)^*\widetilde{\eta} = \iota_L^*\Pi^*\widetilde{\eta}$$ 
throughout $T_{S, L}$, since $\Pi \circ \pi_S \equiv \Pi$ by Corollary \ref{cor-cprsctrl}.\end{proof}

\begin{cor}\label{cor-extindep} The germ of the smooth form $\Pi^*\widetilde{\eta}$ along $X$ does not depend on the choice of the extension $\widetilde{\eta}$ of $\eta$. \end{cor}

\begin{proof}We proceed by upward induction on strata. Let $S \in \Sigma$ be a deepest stratum of $X$. Then as $\partial S = \emptyset$, $\Pi \equiv \pi_S$ on $T_S(\varepsilon)$. Thus, $\iota_S^*\Pi^*\widetilde{\eta} = \iota_S^* \widetilde{\eta} = \eta_S$ defined on $S$ is independent of the choice of $\widetilde{\eta}$. 

Let $L \in \Sigma$ be a stratum and by induction assume that for all boundary strata $S \in \Sigma$, $S \subset \partial L$ of $L$, $\iota_S^* \Pi^*\widetilde{\eta}$ defined on $S$ is independent of the choice of $\widetilde{\eta}$. Then $\iota_L^* \Pi^*\widetilde{\eta} = \pi_{S, L}^* \iota_S^*\Pi^*\widetilde{\eta}$ defined on $T_{S, L}(\varepsilon)$ is also independent of the choice of $\widetilde{\eta}$. Finally, the restriction $\Pi \circ \iota_L : L \setminus \bigcup_{S \subset \partial L} T_{S, L}(\varepsilon) \to L \subset U$ is a smooth map, therefore $\iota_L^*\Pi^*\widetilde{\eta} = \iota_L^* \Pi^* \eta_L$ is independent of the choice of $\widetilde{\eta}$ as well. This shows that $\iota_L^* \Pi^* \widetilde{\eta}$ is independent of the choice of $\widetilde{\eta}$,  concluding the proof. \end{proof}

In light of the previous corollary, we can and will unambiguously denote the smooth differential form $\Pi^*\widetilde{\eta}$ along $X$ as $\Pi^*\eta$. 

\begin{prop}\label{prop-che} $\Pi^* : \Omega^\bullet(X) \to \Omega_{co}^\bullet(X)$ is a chain homotopy equivalence. \end{prop}

\begin{proof}Given a smooth differential form $\eta$ on $X$, we can choose a smooth extension $\widetilde{\eta}$ to some regular neighborhood $U$ of $X \subset M$. For any stratum $S \in \Sigma$ of $X$,
$$\iota_S^* \Pi^*d\widetilde{\eta} = d\iota_S^*\Pi^*\widetilde{\eta}$$
Therefore, $d\Pi^*\eta = \Pi^*d\eta$; hence $\Pi^*$ is indeed a chain map. Next, observe that as $\Pi : U \to U$ is homotopic to the identity, $\Pi^* : \Omega^\bullet(U) \to \Omega^\bullet(U)$ is chain homotopic to the identity. By the natural identification $\Omega^\bullet(U) \cong \Omega^\bullet(X)$, we have that $\Pi^* : \Omega^\bullet(X) \to \Omega^\bullet(X)$ is chain homotopic to the identity. The image of this map lands inside the subcomplex $\Omega_{co}^\bullet(X) \subset \Omega^\bullet(X)$, therefore we obtain the desired chain homotopy equivalence.\end{proof}

\section{Symplectic stratified spaces}\label{sec-sss}

\subsection{Definition and examples}

\begin{defn}\label{def-sympstrat} Let $(X, \Sigma, \mathscr{T})$ be an abstractly stratified space equipped with a smooth structure $j : X \to M$ (see Definition \ref{strat-smooth}, Theorem \ref{strat-natsume}). A smooth $2$-form $\omega = \{\omega_S : S \in \Sigma\} \in \Omega^2(X)$ shall be called a {\bf (stratified) symplectic form} if for every stratum $S \in \Sigma$, $\omega_S$ is a symplectic $2$-form on $S$. \end{defn}

We shall use the term {\it symplectic stratified space} to denote a pair $(X, \omega)$ where $X$ is an abstractly stratified space equipped with a smooth structure, and $\omega \in \Omega^2(X)$ is a symplectic $2$-form. The smooth $2$-form $\omega$ admits an extension $\widetilde{\omega}$ to a smooth $2$-form defined in a neighborhood of $X \subset M$ (cf.\ Remark \ref{rmk-standingnbhd}). Note that we do not require $\widetilde{\omega}$ to be closed and nondegenerate, only that its restrictions to each stratum of $X$ be so. Definition \ref{def-sympstrat} has appeared in the literature  in the context of Marsden-Weinstein reduction theory, and more generally, in the context of symplectic reduction \cite{DJ_singsymp} (cf. \cite{SL_stratsympred}, where an alternative framework for stratified spaces is developed. We discuss the relationship with our definition in Section \ref{sec-sjamaar}). 

Examples of symplectic stratified spaces include (possibly singular) complex projective varieties equipped with a Whitney stratification by quasiprojective subvarieties, with the restriction of the K\"ahler form from the ambient complex projective space. Next note that nondegeneracy is an open condition. Hence, even after a \emph{smooth}  (but not necessarily complex analytic) perturbation of  the embedding of a (possibly singular) complex projective variety in the ambient projective space, we continue to have a symplectic stratified space.

The primary topological obstruction a symplectic $2$-form imposes on a {\it closed manifold} is the associated nonzero cohomology class. For a {\it stratified space} $X$, a priori there is no such cohomology class available in $H^2_{sing}(X; \Bbb R)$ associated to the stratified symplectic form $\omega$. This is because, as discussed in the beginning of Section \ref{sec-stratdr}, there are two flavors of de Rham cohomology available to us, one is smooth de Rham cohomology arising from the complex of $(\Omega^\bullet, d)$ of germs of smooth differential forms along $X \subset M$ and the other is controlled cohomology arising from the subcomplex $(\Omega^\bullet_{co}, d) \subset \bigoplus_{S \in \Sigma} (\Omega^\bullet_S, d_S)$ of controlled differential forms on $(X, \Sigma, \mathscr{T})$. However, 

\begin{enumerate}
\item $\omega$ is rarely a controlled differential form on $X$ as $\omega_S$ is nondegenerate for all $S \in \Sigma$, and controlled differential forms necessarily degenerate near  strata of lower depth.
\item Despite $\omega$ being a smooth differential form on $X$, it is not necessarily a closed smooth differential form as we do not necessarily have a smooth \emph{closed} extension $\widetilde{\omega}$ to an open neighborhood of $X$ in $M$.
\end{enumerate}

We nevertheless manage to extract such a cohomology class by using the Compression Lemma (Proposition \ref{strat-compress}) to compress the smooth, non-controlled, non-closed form $\omega$ into a controlled closed $2$-form $\Pi^* \omega$.

\begin{prop}  Let $(X, \Sigma, \mathscr{T})$ be an abstractly stratified space equipped with a smooth structure $j : X \to M$ inherited from an embedding $j$ into a smooth manifold $M$. Let $\omega = \{\omega_S : S \in \Sigma\} \in \Omega^2(X)$ be a smooth differential $2$-form on $X$ such that $\omega_S$ is closed on $S$ for every stratum $S \in \Sigma$. Then $\Pi^*\omega \in \Omega_{co}^2(X)$ is a closed form.\end{prop}

\begin{proof} This is basically because $\omega$ is stratum-wise a closed $2$-form, and pulling back commutes with exterior derivative. However, some care must be taken as the compression map $\Pi$ does not preserve strata (see Remark \ref{rmk-compskel}). 

We proceed by upward induction on strata. Let $S \in \Sigma$ be the deepest stratum of $X$. Then $\Pi^*\omega_S = \omega_S$ is closed on $S$. Let $L \in \Sigma$ be a stratum and by induction assume that for all boundary strata $S \in \Sigma$, $S \subset \partial L$ of $L$, $\Pi^*\omega_S$ is closed on $S$. As $\Pi^*\omega$ is a controlled form on $X$ by Proposition \ref{form-compress}, $\Pi^*\omega_L = \pi_{S, L}^*\Pi^*\omega_S$ is a closed form on $T_{S, L}(\varepsilon)$ for all boundary strata $S \subset \partial L$. Finally, $\Pi|_L : L \setminus \bigcup_{S \subset \partial L} T_{S, L}(\varepsilon) \to L \subset U$ is a smooth map, therefore, $\Pi^*\omega_L$ is a closed form on $L \setminus \bigcup_{S \subset \partial L} T_{S, L}(\varepsilon)$ as well. This concludes the inductive step, and the proposition follows.\end{proof}

Therefore, $\Pi^*\omega$ gives rise to a well-defined cohomology class in $H^2_{co}(X)$ and consequently a singular cohomology class in $H^2_{sing}(X; \Bbb R)$ by Theorem \ref{thm-veronadr}. We shall call $[\Pi^*\omega] \in H^2_{sing}(X; \Bbb R)$ the cohomology class {\it associated} to the symplectic form $\omega$. 

\begin{defn}\label{def-integral} 

Let $(X,\Sigma, \mathscr{T})$ be an abstractly stratified space. Let $\omega = \{\omega_S : S \in \Sigma\} \in \Omega^2(X)$ be a smooth differential 2-form on $X$ such that $\omega_S$ is closed on $S$ for every stratum $S \in \Sigma$. Then $\omega$ is said to be {\bf integral} if $[\Pi^*\omega] \in H^2(X; \Bbb R)$ is contained in the image of the (change of coefficients) homomorphism $H^2_{sing}(X; \Bbb Z) \to H^2_{sing}(X; \Bbb R)$.

In particular, if $(X, \omega)$ is a symplectic stratified space, we call $\omega$ integral if the associated cohomology class $[\Pi^* \omega] \in H^2_{sing}(X; \Bbb R)$ is contained in the image of the homomorphism $H^2(X; \Bbb Z) \to H^2(X; \Bbb R)$.
\end{defn}

We are now in a position to state the main theorem of this paper. We shall use the following convention: given an abstractly stratified space $(X, \Sigma, \mathscr{T})$ equipped with a smooth structure $j : X \to M$, two $k$-forms $\eta, \xi \in \Omega^k(X)$ shall be said to be \textit{equal on $X$} if $\iota_S^* \eta = \iota_S^* \xi$ for all strata $S \in \Sigma$ of $X$. We stress that this is not the same as saying $\eta, \xi$ are equal as elements of $\Omega^k(X)$, as smooth differential forms on $X$ are really germs of differential forms on $M$ along $X$. Observe, $\eta$ equals $ \xi$ on $X$ if and only if $\eta = \xi + \lambda$ for some $\lambda \in \Omega^2(X)$ such that $\iota_S^* \lambda = 0$ for all $S \in \Sigma$. 

\begin{thm}\label{thm-gt}  Let $(X, \omega)$ be a compact stratified symplectic space where $\omega \in \Omega^2(X)$ is integral.  There exists $N \geq 1$ and an embedding $f : X \to \mathbb{CP}^N$ such that $f^*\omega_N = \omega$ on $X$, where $\omega_N$ is the standard K\"{a}hler form on $\mathbb{CP}^N$. 
	
	More generally, let $(X,\Sigma,\mathscr{T})$ be a compact abstractly stratified space, and 
	$\omega \in \Omega^2(X)$ be  integral. Then there exists $N \geq 1$ and an embedding $f : X \to \mathbb{CP}^N$ such that $f^*\omega_N = \omega$ on $X$, where $\omega_N$ is the standard K\"{a}hler form on $\mathbb{CP}^N$. 
\end{thm}

Theorem \ref{thm-gt} generalizes a theorem of Gromov (the isosymplectic embedding theorem  \cite{gromov-icm70}) and Tischler \cite{tischler}. The proof occupies the rest of this section and consists of the following steps:\\

\noindent {\bf Step 1:} We would first like to construct a map $f: X \to \Bbb{CP}^N$ such that 
\begin{equation*}``[f^\ast (\omega_N)] = [\omega]",\end{equation*}
 in a suitable sense. Note that by the discussion preceding Theorem \ref{thm-gt},
$[\omega]$ does not  make sense literally. What we really construct is a smooth map $f : U \to \Bbb{CP}^N$ from a neighborhood $U \subset M$ of $X$ such that
\begin{equation*}[f^\ast (\omega_N)] =[\Pi^*\omega],\end{equation*}
where $\Pi$ is the map occurring in the Compression Lemma (Proposition \ref{strat-compress}). The left hand side is the smooth de Rham class of the smooth closed form $f^*(\omega_N)$ on $U$, while the right hand side is the controlled de Rham class of the controlled closed form $\Pi^*\omega$, and the equality makes sense in light of the fact that smooth de Rham cohomology and controlled de Rham cohomology are both isomorphic to the real singular cohomology of $X$ (see Section \ref{sec-stratdr}). This step is carried out in Section \ref{sec-cohclass}.\\

\noindent {\bf Step 2:} We start with the function
$f: U \to \Bbb{CP}^N$  obtained in Step 1, and consider the `stabilization' $f_1 : U \to \Bbb{CP}^N \subset \Bbb{CP}^{N+k}$.  We would like to homotope
$f_1$ to $g : U \to \Bbb{CP}^{N+k}$, such that  $g^\ast (\omega_N) =\omega$ on $X$. This is accomplished in two stages. We first write 
\begin{equation*}f^\ast \omega_N = \omega + d \eta + \beta,\end{equation*}
where
\begin{enumerate}
\item $\eta$ is a smooth 1-form in the neighborhood $U \subset M$ of $X$, and
\item $\beta$ is a smooth 2-form in $U \subset M$ vanishing on each stratum of $X$.
\end{enumerate}
Using a technique going back to Nash, we modify $f_1$ to an embedding $g : U \to \Bbb{CP}^{N+k}$, so that 
\begin{equation*}g^\ast \omega_{N+k} = \omega + \beta\end{equation*}
on $U$. Roughly speaking, we absorb the $d \eta$ term into $g$ by `corrugating' $f_1$ to increase area. Finally, using the fact that $\beta$ vanishes on strata of $X$, we conclude that $$g^\ast \omega_{N+k} \ {\rm equals} \ \omega \ \text{on} \ X$$
in the sense of the convention preceding Theorem \ref{thm-gt}. This is done in Section \ref{sec-ngt}.

\subsection{Equating second cohomology classes}\label{sec-cohclass}

Let $X$ be a smooth stratified symplectic space. Recall that this means there exists a smooth manifold $M$, a smooth embedding $X \subset M$, an open neighborhood $U$ of $X$, and a 2-form $\omega$ on $U$ such that $\omega\vert_S$ is closed non-degenerate for every stratum $S$. Equip $U$ with the stratification given by $U \setminus X$ along with the strata of $X$. 
Let $\Pi$ be the compression map occurring in the Compression Lemma (Proposition \ref{strat-compress}, for the stratification on $U$. 
By Proposition \ref{strat-compress},  $\Pi^*\omega$ is a smooth closed form on $X$ and hence $[\Pi^*\omega]$ is well-defined. Without loss of generality, we assume henceforth that $U \subset X$ is a $\varepsilon$-neighborhood of $X$ where $\varepsilon > 0$ is chosen sufficiently small so that $\Pi(U \setminus X) = X$. Then $\Pi^*\omega$ is a controlled closed form on the entire stratified space $U$.

\begin{prop}\label{prop-formalsoln} Suppose that $\omega$ is integral. There exists $N \in \natls$ and an embedding $f: U \to \Bbb{CP}^N$, such that
\begin{equation}\label{eq-classeq}[f^\ast (\omega_N)]=[\Pi^*\omega],\end{equation}
on $U$, where $\omega_N$ is the K\"ahler form on $\Bbb{CP}^N$. In particular, Equation \eqref{eq-classeq} holds on $X$.
\end{prop}

\begin{proof} Note that 
$[\Pi^*\omega]$ is a well-defined cohomology class in $H^2_{co}(U) = H^2_{co} X)$.
By Theorem \ref{thm-veronadr}, $H^2_{co}(U) \cong H^2_{sing}(U, \R)$. Further, since $\om$ is integral,
$[\Pi^*\omega]$ lies in the image of $H^2_{sing} (U, \Z)$ in $H^2_{sing} (U, \R)$.

Next, observe that $H^2_{sing}(U, \Z) \cong [U, \Bbb{CP}^\infty]$ as $\Bbb{CP}^\infty \simeq K(\Bbb Z, 2)$ is an Eilenberg-Maclane space, which represent singular cohomology. Thus, we have the following sequence of maps:
$$[U, \Bbb{CP}^\infty] \stackrel{\cong}\longrightarrow H^2_{sing} (U, \Z)  \rightarrow H^2_{sing} (U, \R)
\stackrel{\cong} \longrightarrow H^2_{dR} (U) \stackrel{\cong} \longrightarrow H^2_{co} (U),$$
where the second arrow is injective and the remaining maps are isomorphisms. The third arrow is the
de Rham isomorphism, and the fourth arrow is the pullback by the compression map, sending $[\eta]$ to $[\Pi^*\eta]$ for any closed form $\eta$. Let $\alpha \in H^2_{sing} (\C P^\infty, \Z)$ denote the generator of the cohomology ring. Then the first arrow
sends $f \in [U, \C P^\infty]$ to $[f^* \alpha]$. 

As $U$ is a manifold, it has the homotopy type of a finite-dimensional CW-complex. By the cellular approximation theorem, $[U, \Bbb{CP}^\infty] \cong [U, \Bbb{CP}^N]$ for some sufficiently large natural number $N$. Under this identification, the first arrow $[U, \Bbb{CP}^N] \cong H^2_{sing}(U, \Bbb Z)$ can be described as $f \mapsto f^*(\omega_N)$ where $\omega_N$ is the K\"ahler form on $\Bbb{CP}^N$. Thus, the composition of the first three arrows is given by $f \mapsto [f^*\omega_N]$.

On the other hand, the integral cohomology class $[\Pi^*\om] \in \mathrm{im}(H^2_{sing}(U; \Bbb Z) \to H^2_{sing} (U, \R))$ defines a homotopy class of maps $U \to \Bbb{CP}^N$ for some large $N$. Choosing any $f$ from this homotopy class,
we have 
$$[f^*\omega_N]=[\Pi^*\om].$$
Choosing $N$ sufficiently large, and homotoping $f$ slightly if necessary, we can assume without loss of generality that $f: U \to \C P^N$ is a smooth embedding. This proves the first assertion of the proposition.

Similarly, we also have:
$$ [X, \C P^\infty] \stackrel{\cong}\longrightarrow H^2_{sing} (X, \Z)  \rightarrow H^2_{sing} (X, \R)
\stackrel{\cong} \longrightarrow H^2_{dR} (X, \R) \stackrel{\cong} \longrightarrow H^2_{co} (X, \R),$$
and the maps are natural with respect to the inclusion $X \subset U$, proving the 
second  assertion of the proposition.\end{proof}

\subsection{The Nash-Gromov-Tischler embedding}\label{sec-ngt}
In Proposition \ref{prop-formalsoln} we obtained what can be described as a {\it formal solution} to the problem of finding an \emph{isosymplectic embedding} of $X$ in $\Bbb{CP}^N$ in the sense of \cite[Section 12.1.1]{em-book}, i.e.\ a map $f : X \to \Bbb{CP}^N$ satisfying $f^*[\omega_N] = [\Pi^* \omega]$. In this section we upgrade $f$ to a {\it holonomic solution}, which in this case would be an honest isosymplectic embedding $g : X \to \Bbb{CP}^m$ for some $m \geq 1$, i.e., $g^*\omega_m = \omega$.

We begin by observing that $f^*\omega_N = \Pi^* \omega + d\alpha'$ for some $\alpha' \in \Omega^1(X)$. Let $\widetilde{\omega}$ be a smooth extension of $\omega$ to some regular neighborhood $U$ of $X \subset M$. Recall $\Pi : U \to X$ is the time-$1$ map of a homotopy $\mathcal{H} : U \times I \to U$. We use Cartan's magic formula to integrate over the fibers of this homotopy,
\begin{align*}\Pi^*\widetilde{\omega} - \widetilde{\omega} = \mathcal{H}(-, 1)^*\widetilde{\omega} - \mathcal{H}(-, 0)^*\widetilde{\omega} &= \int_0^1 \partial_t(\mathcal{H}^*\widetilde{\omega}) \,dt
\\ &= \int_0^1 \mathcal{L}_{\partial_t} \mathcal{H}^*\widetilde{\omega} \, dt
\\ & = \int_0^1 \left ( d i_{\partial_t} \mathcal{H}^*\widetilde{\omega} +  i_{\partial_t} d \mathcal{H}^*\widetilde{\omega} \right)\, dt
\\ &= d\alpha'' + \beta\end{align*}
where we have set
\begin{gather*}\alpha'' = \int_0^1 i_{\partial_t} \mathcal{H}^*\widetilde{\omega}\, dt  \\ \beta = \int_0^1 i_{\partial_t} d\mathcal{H}^*\widetilde{\omega}\, dt\end{gather*} 
Now, observe,
\begin{align*}\beta_p(v_1,v_2) &= \int_0^1 d\mathcal{H}^*\widetilde{\omega}_p(\partial_t, v_1, v_2) dt
\\ &= \int_0^1 d\widetilde{\omega}_q(\mathcal{H}_*(\partial_t), \mathcal{H}_*(v_1), \mathcal{H}_*(v_2)) dt \end{align*}
where $q = \mathcal{H}(p)$. Suppose $p \in X \times 0 \subset U \times I$ and $v_1, v_2 \in T_p S$ where $S \in \Sigma$ is the unique stratum containing $p$. Then $q \in X$ and as $\mathcal{H} : U \times I \to U$ is a smooth map, $\mathcal{H}_*(v_1), \mathcal{H}_*(v_2), \mathcal{H}_*(\partial_t) \in T_q S'$, where $S' \in \Sigma$ is the unique stratum containing $q$. As $\omega$ is closed restricted to every stratum, we obtain that the integrand is zero, therefore $\iota_S \beta = 0$ for every stratum $S \in \Sigma$ of $X$.
Combining everything, we obtain
$$f^*\omega_N = \omega + d\alpha + \beta$$
where $\alpha = \alpha' + \alpha'' \in \Omega^1(U)$ is a smooth $1$-form and $\beta \in \Omega^2(U)$ is a smooth $2$-form which vanishes along $X$, in the sense that $\iota_S^* \beta = 0$ for every stratum $S \in \Sigma$.\\

We shall try to modify $f$ to an embedding $g : U \to \Bbb{CP}^m$ such that $g^*\omega_m = \omega + \beta$ on $U$. Roughly, our strategy is to first {\it stabilize} $f$ to $f : U \to \Bbb{CP}^{N+1}$ by composing with the canonical inclusion $\Bbb{CP}^N \hookrightarrow \Bbb{CP}^{N+1}$ for some $k \geq 1$ to gain ``enough space" so as to {\it corrugate} $f$ (in the spirit of \cite{em-book}) within a normal neighborhood of $\Bbb{CP}^N$ in $\Bbb{CP}^{N+1}$, and obtain a new map $f_1 : U \to \Bbb{CP}^{N+1}$. The purpose of these corrugations is to compensate, at in least in part, for the ``area-defect" $d\alpha$ in $f^*\omega_N - \omega$. We repeat this process many times to finally obtain an isosymplectic embedding $g : U \to \Bbb{CP}^{N+k}$. 

Let us carry out the above sketch in detail. First, select a finite open cover $\{V_1, \cdots, V_\ell\}$ of $\Bbb{CP}^N$ such that the normal bundle $\mathcal{O}(1)$ of $\Bbb{CP}^N \subset \Bbb{CP}^{N+1}$ trivializes over each element of the open cover. We pullback this open cover to $U$ by $f$, and pass to a finer collection $\{W_1, \cdots, W_k\}$ of charts of $M$ covering $X$ such that $f(W_i) \subset V_{j}$ for some $1 \leq j \leq \ell$, using compactness of $X$. \emph{This is the only place where we need compactness of $X$.}  By passing to an even finer finite collection covering $X$ if necessary, we can assume that there exist compactly supported smooth functions $p_i, q_i : U \to \Bbb R$ with $\supp(p_i), \supp(q_i) \subset W_i$ such that,
$$\alpha = p_1 dq_1 + \cdots + p_k dq_k$$
Indeed, to see that this refinement is possible, note that by using a partition of unity subordinate to $\{W_1, \cdots, W_k\}$, it suffices to check that such a decomposition is possible for a $1$-form defined on any one of the charts $W_i \subset U$. Equivalently, it suffices to check this for any $1$-form on an open subset of a Euclidean space, where it is clear. We conclude the following:
$$f^*\omega_N = \omega + \sum_{i = 1}^k dp_1 \wedge dq_1 + \beta$$

Next, let us stabilize $f$ by a single extra dimension to $f : U \to \Bbb{CP}^N \subset \Bbb{CP}^{N+1}$. Focus on a single chart $W_i$, and suppose that $f(W_i) \subset V_j$, where $1 \leq j \leq \ell$. By our choice of $V_j$, we have a trivialization of the normal bundle, $\mathcal{O}(1)|_{V_j} \cong V_j \times \Bbb C$. By the symplectic tubular neighborhood theorem, there is a symplectic embedding of an appropriate disk bundle of $\mathcal{O}(1)$ in $\Bbb{CP}^{N+1}$ as a normal neighborhood of $\Bbb{CP}^N$. In particular we obtain a symplectic embedding $V_j \times \Bbb D(\varepsilon) \hookrightarrow \Bbb{CP}^{N+1}$ as a piece of the normal neighborhood lying over $V_j \subset \Bbb{CP}^N$, for some $\varepsilon > 0$.

Using Darboux coordinates, there exists a smooth map $h_i : W_i \to \Bbb D(\varepsilon)$ such that
\begin{equation}\label{eq-darboux} h_i^*(dz \wedge d\bar{z}) = dp_i \wedge dq_i,\end{equation}
where $\Bbb D(\varepsilon)$ is equipped with a complex coordinate $z$.
Using Equation \eqref{eq-darboux}, we can construct a "corrugation" of $f$ as follows. Define $f_1 : U \to \Bbb{CP}^{N+1}$ by $f_1(x) = (f(x), -h_1(x))$ for all $x \in W_1$ and $f_1 \equiv f$ outside of $W_1$. Observe, 
$$f_1^*(\omega_{n+1}) = f^*\omega_N - h_i^*(dz \wedge d\bar{z}) = \omega + \sum_{i = 2}^k dp_1 \wedge dq_1 + \beta.$$
We repeat this process $k$ times to obtain $g := f_k : U \to \Bbb{CP}^{N+k}$ with $$g^*\omega_{n+k} = \omega + \beta.$$
Since $\beta\vert_S = 0$ for every stratum $S \in \Sigma$, this completes the proof of Theorem \ref{thm-gt}. \hfill $\Box$\\

The \emph{proof} of Theorem \ref{thm-gt} above gives the following $h$-principle result for symplectic embeddings in complex projective spaces.

\begin{cor}\label{cor-hat}  Let $(X, \Sigma, \mathscr{T})$ be an abstractly stratified space equipped with a smooth structure $j : X \to M$. Let $\omega \in \Omega^2(X)$ be a smooth $2$-form and $f : X \to \Bbb{CP}^N$ be an embedding such that $f^* [\omega_N] = [\Pi^*\omega]$. Then there exists $k \geq 1$ and an embedding $g : X \to \Bbb{CP}^{N+k}$ such that $g^*\omega_{n+k} = \omega$ on $X$, and moreover $g$ is $\varepsilon$-close to $f : X \to \Bbb{CP}^N \subset \Bbb{CP}^{N+k}$ in the $C^0$-topology.\end{cor}

In particular, if $[\Pi^*\omega] = 0$, then the constant map $f : X \to \Bbb{CP}^0$ satisfies the hypothesis of Corollary \ref{cor-hat} above. Hence, by Corollary \ref{cor-hat}, we obtain a symplectic embedding $g : X \to \Bbb{CP}^k$ such that $g$ is $\varepsilon$-close to the constant map to a point $p := \Bbb{CP}^0 \subset \Bbb{CP}^k$, and therefore has image contained in the affine open chart $\Bbb{R}^{2n} \subset \Bbb{CP}^0$ containing $p$. We conclude the following.

\begin{cor}\label{cor-0impliesrn} Let $(X, \Sigma, \mathscr{T})$ be an abstractly stratified space equipped with a smooth structure $j : X \to M$. Suppose $\omega \in \Omega^2(X)$ is a smooth $2$-form such that $[\Pi^*\omega] = 0$. Then there exists an embedding $f : X \to \Bbb R^{2N}$ such that $f^* \omega^0_N = \omega$ where $\omega^0_N$ is the standard symplectic form on $\Bbb R^{2N}$. 
\end{cor}

\section{Algebraic and geometric notions of symplectic stratified spaces}\label{sec-sjamaar}

\subsection{Sub and quotient algebras of functions}\label{sec-eqqnts} 
We start with the following theorem of Schwarz \cite{schwarz-Rn} and Mather \cite{mather-diffinv}
as summarized in \cite[Theorem 1.1]{herbig-pflaum}.
\begin{thm}\label{thm-schwarzmather}
Let $G \to O(V )$ be a finite dimensional orthogonal representation
of the compact Lie group $G$. Let $\rho_1\cdots,\rho_l \in  \R[V ]^G$ be a minimal complete system
of polynomial invariants. Let $\rho = (\rho_1\cdots,\rho_l) : V \to \R^l$ be the corresponding Hilbert
map. Then the pullback
$\rho^* : C^\infty (\R^l) \to C^\infty(V )^G$
is split surjective, i.e.\ there exists a continuous map $\lambda : C^\infty(V )^G \to 
C^\infty (\R^l)$ such that $\rho^* \circ \lambda
:C^\infty(V )^G \to C^\infty(V )^G$ is the identity.
\end{thm}

The map $\rho$ factors through a quotient map $i: V/G \to \R^l$, which is a proper topological embedding as shown in \cite[Section 3]{mather-diffinv}. The embedding $i$ is referred to as the {\it Hilbert embedding} in \cite{herbig-pflaum}. We shall follow this terminology.

Let $X = i(V/G)$. We equip $V$ with the decomposition by orbit type, which descends to a decomposition of $V/G \cong X$ as well. This decomposition of $X \subset \Bbb R^l$ is in fact a Whitney stratification, as alluded to in the remarks following the Proof of Theorem 6.7 in \cite[pg. 42]{SL_stratsympred}. Briefly, here are the key ideas: The strata of $V$ are semialgebraic, and since the map $\rho$ is algebraic, their images are semialgebraic subsets of $\Bbb R^l$ by the Tarski-Seidenberg theorem. Therefore, the strata of $X$ are semialgebraic subsets of $\Bbb R^l$. To check that a pair of strata $(S, L)$ of $X$ satisfies the Whitney condition $(b)$, one applies the slice theorem at a point $x \in S$ to reduce to the case where $S = \{x\}$ is a singleton. By Whitney's theorem (see the discussion on p.\ 21 of \cite{gwpl} for instance), the locus of points of $S$ where a pair of semialgebraic smooth manifolds $(S, L)$ fails to satisfy Whitney condition $(b)$ is a semialgebraic set of dimension strictly less than $\dim(S)$, hence in the case where $S$ is a singleton it must be empty. This proves the result.

As explained in \cite{herbig-pflaum}, the orbit space $V/G \cong X$ can thus be equipped with a ``smooth structure"  in the sense of \cite[pg. 7]{SL_stratsympred} (which is \emph{not} to be confused with our notion in Definition \ref{strat-smooth}) in two different ways:

\begin{enumerate}
\item By setting $C^\infty(V/G) = C^\infty(V)^G$ as the algebra of smooth functions on $X$.
\item By setting $C^\infty(X) = \{f \in C^0(X) \vert \exists F \in C^\infty(\R^l ) \ \text{s.t.} \ F\vert_X = f \}$ as the algebra of smooth functions on $X$.
\end{enumerate}

\begin{cor}\label{cor-isomRn}\cite[pg. 2]{herbig-pflaum}
The Hilbert embedding $i$ induces an isomorphism of algebras 
$$i^*: C^\infty(X) \to C^\infty(V)^G$$
If $Z \subset V$ is a closed $G$-invariant subset, then likewise 
$$i^*: C^\infty(\rho(Z) ) \to C^\infty(Z)^G$$
is an isomorphism of algebras, where $C^\infty(\rho(Z)), C^\infty(Z)$ are the subalgebra of continuous functions on $\rho(Z), Z$ which admit a smooth extension to $\Bbb R^l, V$, respectively.
\end{cor}

A word about the last statement of Corollary \ref{cor-isomRn} above is in order (see Remark 4.2 and Proposition 4.4 of \cite{herbig-pflaum} for a more general statement). This is because $C^\infty(\rho(Z))$ can be identified with the quotient of $C^\infty(X)$ by the ideal $\II_{\rho(Z)}$ of functions that vanish on $\rho(Z)$. The image of this ideal $i^* (\II_{\rho(Z)}) \subset C^\infty(V)^G$ consists of precisely the $G-$invariant functions on $V$ that vanish on $Z$, hence 
$$C^\infty(V)^G/i^*(\II_{\rho(Z)}) \cong C^\infty(Z)^G$$
Indeed, for any $G$-invariant function $f \in C^\infty(Z)$, if $F \in C^\infty(V)$ is a smooth extension of $f$, averaging with respect to the Haar measure on the compact Lie group $G$, we obtain a $G$-invariant extension $\widetilde{F} \in C^\infty(V)^G$ of $f$. Thus, the restriction map $C^\infty(V)^G \to C^\infty(Z)^G$ is surjective and the kernel of this map is exactly the ideal $i^*(\II_{\rho(Z)})$. 

Therefore $i^*$ induces an isomorphism $C^\infty(\rho(Z)) \cong C^\infty(Z)^G$, as required.

\begin{rmk}\label{rmk-subanalytic} The main theorem of \cite{herbig-pflaum} is a generalization of Corollary \ref{cor-isomRn} where $Z$ is not assumed to be a $G$-invariant subset. This requires assuming that $Z$ is subanalytic, which is not necessary for our purposes.\end{rmk}

Let $(M,\om)$ be a smooth symplectic manifold equipped with a Hamiltonian $G-$action. Let $\mu: M \to \mathfrak{g}^*$ be the associated moment map. If $G$ acts on $M$ with finitely many orbit types, then by the Mostow--Palais theorem
\cite{mostow-mp,palais-mp} there exists a $G-$equivariant smooth embedding $\phi : M \to \Bbb R^n$ where $G$ acts by some orthogonal representation on $\Bbb R^n$. Let $\rho : \Bbb R^n \to \Bbb R^l$ be the Hilbert map for this action, which factors through the Hilbert embedding $i : \Bbb R^n/G \to \Bbb R^l$.

Let $Z = \mu^{-1}(0) \subset M$ so that $Z/G = M\sslash G$ is the symplectic reduction. Let $X = \rho(\phi(Z))$ be the embedded image of $Z/G$ in $\Bbb R^l$. Let $C^\infty(Z), C^\infty(X)$ denote the subalgebra of continuous functions on $Z, X$ which extend to smooth functions on $M, \Bbb R^l$, respectively. 

\begin{prop}\label{prop-mu}
The restriction $i : Z/G = M\sslash G \to X$ induces an isomorphism of algebras $i^* : C^\infty(X) \to C^\infty(Z)^G$. 
\end{prop}

\begin{proof}
We apply Corollary \ref{cor-isomRn} to $Z = \mu^{-1}(0)$. 
\end{proof}

\subsection{Two notions}\label{sec-def12} We first recall Definition \ref{def-sympstrat} giving a 
{\it geometric} notion of a symplectic stratified space: let $(X, \Sigma, \mathscr{T})$ be a stratified space equipped with an embedding $j : X \to M$ into a smooth manifold $M$ defining a smooth structure on $X$. Then,
as per Definition \ref{def-sympstrat}, a smooth $2-$form $\omega = \{\omega_S : S \in \Sigma\} \in \Omega^2(X)$ is a  (stratified) symplectic form if for every stratum $S \in \Sigma$, $\omega_S$ is a symplectic $2-$form on $S$. Note that $\om$ admits an extension $\til \om$ to an open neighborhood of $X$ in $M$. For the purposes of this section, we shall refer to $X$ equipped with
a (stratified) symplectic form as a {\it geometric symplectic stratified space}. 

In 
\cite[Definition 1.12]{SL_stratsympred}, Sjamaar and Lerman  (see also \cite{sjamaar-lms,pflaum-expo,slv}) 
provided a different definition of a symplectic stratified space, specifically in the context of symplectic
reduction. Specializing their definition by choosing an appropriate algebra of smooth functions on the stratified space, we have the following notion of a Poisson--symplectic stratified space, where we adjoin the prefix `Poisson' to distinguish it from the geometric notion of Definition \ref{def-sympstrat}.

\begin{defn}\label{def-alsss}
Let $(X, \Sigma, \mathscr{T})$ be a stratified space equipped with an embedding $j : X \to M$ into a smooth manifold $M$ defining a smooth structure on $X$. Let $C^\infty(X)$ denote the algebra of continuous functions on $X$ which extend to a smooth function on $M$. Then $X$ is said to be a {\it Poisson--symplectic stratified space} if

\begin{enumerate}
\item Each stratum $S$ is equipped with a symplectic $2-$form $\omega_S$. Let $\{\cdot \, , \, \cdot \}_S$ denote the induced Poisson bracket on $C^\infty(S)$.
\item $C^\infty(X)$ is equipped with a Poisson algebra structure. Let $\{\cdot \, , \, \cdot \}$ denote the Poisson bracket on $C^\infty(X)$. 
\item The inclusions $S \to X$ are Poisson embeddings, i.e.\ the Poisson bracket induced on the restrictions $\{f\vert_S, f \in C^\infty(X)\}$ is given by $\{\cdot \, , \, \cdot \}_S$. Equivalently, the restriction map 
  $$\big( C^\infty(X), \{\cdot \, , \, \cdot \} \big)
 \to \big( C^\infty(S), \{\cdot \, , \, \cdot \}_S \big)$$ is a Poisson map.
 \end{enumerate}
\end{defn}

\subsection{Algebraic implies geometric}\label{sec-alimpliesgeo}

Recall that a nondegenerate bilinear form $B : V \times V \to \Bbb R$ on any vector space defines an isomorphism $B : V \to V^*$ and its inverse $B^{-1} : V^* \to V$.
In particular, for a symplectic manifold $(N, \omega)$, the symplectic form 
$\omega$ gives an isomorphism $\omega_n : T_n N \to T_n^*N$ for every $n \in N$ and hence an inverse $\omega_n^{-1} : T_n^*N \to  T_n N$. We denote by $\omega^{-1}$ the smoothly parametrized family $\{\omega_n^{-1}\}$
  as $n$ ranges over $N$.

\begin{defn}Let $(N, \omega)$ be a symplectic manifold equipped with a Riemannian metric $g$. We define $\mathrm{Inv}_g(\omega)$ to be the $2$-form on $N$ corresponding to the composition
$$TN \stackrel{g}{\longrightarrow} T^*N \stackrel{\omega^{-1}}{\longrightarrow} TN \stackrel{g}{\longrightarrow} T^*N$$
Equivalently, for any pair of vector fields $X, Y$, 
$$\mathrm{Inv}_g(\omega)(X, Y) := g(Z, Y)$$ 
where $Z$ is the vector field uniquely defined by $\omega(Z, -) = g(X, -)$.\end{defn}

Observe that the matrix of $\mathrm{Inv}_g(\omega)$ with respect to an orthonormal basis at a point is the \emph{inverse} of the matrix of $\omega$, hence the terminology. In the following proposition, we show that the condition of Poisson--symplecticity on a stratified space is equivalent to extension of the collection of $2$--forms on each stratum given by inverting the stratumwise symplectic forms, to a smooth $2$--form on the ambient manifold. 

\begin{prop}\label{prop-alimpliesgeo}
Let $(X, \Sigma, \mathscr{T})$ be a Poisson--symplectic stratified space equipped with a smooth structure $j : X \to M$. Let $g$ be an ambient Riemannian metric on $M$. Let $\{\omega_S : S \in \Sigma\}$ denote the collection of symplectic forms. Then, the collection of $2$--forms
$$\{\mathrm{Inv}_g(\omega_S) : S \in \Sigma\}$$
admits an extension to a smooth $2$-form $\theta$ on $M$, i.e., $i_S^*\theta = \mathrm{Inv}_g(\omega_S)$ for all $S \in \Sigma$, where $i_S : S \hookrightarrow M$ is the inclusion map.
\end{prop}

\begin{proof}Let $\phi : M \to \Bbb R^n$ be a smooth embedding in a Euclidean space and let $x_1, \cdots, x_n \in C^\infty(\Bbb R^n)$ be the standard coordinate functions on $\Bbb R^n$. Restricted to $M$, and then to $X$, these give rise to functions $x_1, \cdots, x_n \in C^\infty(X)$. By Condition $(2)$ and $(3)$ in Definition \ref{def-alsss}, for any $1 \leq i, j \leq n$, the collection of functions 
$$\{\{x_i, x_j\}_S : S \in \Sigma\}$$
defined on each strata extends to a smooth function on $M$. That is, there exists $f_{ij} \in C^\infty(M)$ such that $f_{ij}|_S = \{x_i, x_j\}_S$ for all $S \in \Sigma$. $f_{ij}|_X$ would be the Poisson bracket of $x_i$ and $x_j$ in $C^\infty(X)$ as in Condition $(2)$.

Let $H_{S, i}$ denote the Hamiltonian vector field corresponding to $x_i$ in $(S, \omega_S)$, defined by $\omega_S(H_{S, i}, V) = dx_i(V)$ for any vector field $V$ on $S$. Then,
$$\{x_i, x_j\}_S = \omega_S(H_{S, i}, H_{S, j})$$
For any stratum $S \in \Sigma$, we may extend $\omega_S$ to a (not necessarily properly) embedded tubular neighborhood $\nu_S$ of $S$ in $\Bbb R^n$, and therefore write
$$\omega_S = \sum_{1 \leq i < j \leq n} g_{ij} dx_i \wedge dx_j$$
for smooth functions $g_{ij}$ defined on $\nu_S$. Let $\Omega_S$ denote the $n \times n$ antisymmetric matrix of functions defined on $\nu_S$ given uniquely by $(\Omega_S)_{i, j} = g_{ij}$ if $i < j$. Then 
$$\omega_S(v, w)_p = \langle\Omega_S(p)(v), w\rangle$$
for any $p \in S, v, w \in T_p S \subset T_p \Bbb R^n$ where on the right hand side we write $v, w \in T_p \Bbb R^n$ in terms of the standard basis $\partial/\partial x_1, \cdots, \partial/\partial x_n$, and $\langle \cdot, \cdot \rangle$ is the standard inner product on $\Bbb R^n$. In light of this, $\omega_S(H_{S, i}, V) = dx_i(V)$ implies 
\begin{align*}
\langle \Omega_S (H_{S, i}), v\rangle &= dx_i(v) = \langle \partial/\partial x_i, v\rangle 
\end{align*}
for all tangent vectors $v$ to $S$. Thus, $\Omega_S(H_{S, i}) = \partial/\partial x_i + Y_{S, i}$ where $Y_{S, i}$ is a normal vector field to $S$ in $M$, with respect to the Riemannian metric induced from the embedding $\phi : M \to \Bbb R^n$.

Let us treat $\Omega_S$ as a bundle endomorphism of $T\Bbb R^n|_S$. If $P_S : T\Bbb R^n|_S \to TS$ denotes the fiberwise orthogonal projection, and $Q_S : TS \to T\Bbb R^n|_S$ denotes the fiberwise inclusion, then we have 
$$(P_S \Omega_S Q_S)(H_{S, i}) = P_S(\partial/\partial x_i)$$
As $\omega_S$ is nondegenerate on $S$, $P_S \Omega_S Q_S : TS \to TS$ is a bundle isomorphism. Hence, 
$$H_{S, i} = (P_S \Omega_S Q_S)^{-1}(P_S(\partial/\partial x_i))$$
Observe that $(P_S \Omega_S Q_S)^{-1}P_S$ is a bundle endomorphism $T\Bbb R^n|_S \to TS$. Composing with $Q_S : TS \to T\Bbb R^n|_S$, we write 
$$\Theta_S := Q_S(P_S\Omega_S Q_S)^{-1}P_S$$
which is a bundle endomorphism of $T\Bbb R^n|_S$. Therefore, $H_{S, i} = \Theta_S(\partial/\partial x_i)$. Consequently, we have
$$\{x_i, x_j\}_S = \omega_S(H_{S, i}, H_{L, j}) = dx_i(H_{S, j}) = \langle \Theta_S(\partial/\partial x_j), \partial/\partial x_i \rangle = (\Theta_S)_{i, j}$$
Therefore, the condition of being a Poisson--symplectic stratified space implies the collection of functions $\{(\Theta_S)_{i, j} : S \in \Sigma\}$ defined on each strata extend to smooth functions $f_{ij}$ on $M$, for all $1 \leq i, j \leq n$.

Let us define the following $2$-form in a tubular neighborhood $\nu_S$ of $S$ in $M$:
$$\eta_S := \sum_{1 \leq i, j \leq n} (\Theta_S)_{i, j} dx_i \wedge dx_j$$
Then, for any $1 \leq j \leq n$, $\eta_S(\partial/\partial x_j, -) = g((\Theta_S)(\partial/\partial x_j), -)$ on $S$. Hence, $\eta_S(X, Y) = g(Z, Y)$ where $Z = \Theta_S(X)$. On the other hand, $\omega(\Theta_S(\partial/\partial x_j), -) = \omega(H_{S, j}, -) = dx_j = g(\partial/\partial x_j, -)$. Hence, $\omega(Z, -) = g(X, -)$. Thus, $i_S^* \eta_S = \mathrm{Inv}_g(\omega_S)$ where $i_S : S \hookrightarrow M$ is the inclusion map. Finally, define
$$\theta := \sum_{1 \leq i, j \leq n} f_{ij} dx_i \wedge dx_j$$
Then $\theta$ is a smooth $2$-form on $M$ such that $i_S^*\theta = i_S^*\eta_S = \mathrm{Inv}_g(\omega_S)$, as desired.
\end{proof}

\begin{ex}\label{ctreg}{\rm 
		Consider  $\R^4=\R^2 \times \R^2$ as a stratified space with strata $S=\R^2 \times \{0\}$ and
		$L=\R^2 \times
		\R^2\setminus \{0\}$. Let $\omega_1 = dx_1 \wedge dy_1$ on the first $\R^2$ factor, and $\omega_2 = dx_2 \wedge dy_2$ on the second $\R^2$ factor. 
		Equip $\R^4$ with the 2-form $\omega=\omega_1 + r^2 \omega_2 $, where $r^2 = x_2^2 + y_2^2$.
		Then 
		\begin{enumerate}
			\item $\omega_S = \omega_1,\, \,  \omega_L= \omega$ does  furnish a geometric symplectic stratified space structure on $X=\R^4$ stratified as above.
			\item $\omega_S = \omega_1, \omega_L= \omega$ does not furnish a Poisson-symplectic stratified space
			structure on $X=\R^4$ stratified as above.
	\end{enumerate}}
\end{ex}

\begin{ex}{\rm 
	Folded symplectic structures as in \cite{cds} give examples of geometric symplectic stratified spaces in the sense of Definition
	\ref{def-sympstrat}. It is not hard to see that these are not examples of Poisson-symplectic stratified spaces.
	
	On the other hand,	b-symplectic structures as in \cite{gmp} give examples of Poisson-symplectic stratified spaces in the sense of Definition \ref{def-alsss}. However, these are not examples of geometric symplectic stratified spaces in the sense of Definition
	\ref{def-sympstrat}.
}
\end{ex}

\subsection{Geometric implies  algebraic }\label{sec-geoimpliesal} In the converse direction to Proposition \ref{prop-alimpliesgeo} we have the following conclusion. We adopt a convention followed by Gromov \cite[Ch. 1]{Gromov_PDR}. Let $M$ be a manifold and $K \subset M$ denote a a closed subset. 
Then the germ of an open neighborhood of  $K$ in $ M$ will be denoted as $\OO p (K)$. We define the \emph{open restriction} $C^{\infty}(\OO p (K))$
of $C^{\infty}(M)$ to $\OO p (K)$ to be the quotient of  $C^{\infty}(M)$ under the equivalence relation
$f \sim g$, if $f, g$ agree on an open neighborhood of $K$.

\begin{prop}\label{prop-geoimpliesal}
Let $(X, \Sigma, \mathscr{T})$ be a geometric symplectic stratified space (in the sense of Definition \ref{def-sympstrat}).
Then $C^{\infty}(\OO p (X))$ is a centerless Poisson algebra.
\end{prop}

\begin{proof} Let $\om=\{\om_S:S \in \Sigma\}$ be a collection of symplectic 2-forms on 
	the strata $S \in \Sigma$, defining a geometric symplectic stratified space structure on $X$.
	Let $H^2(X,\R) = \R^M$. Let $\Pi$ denote the compression map occurring in Proposition \ref{form-compress}. Let $[\Pi^\ast \om] \in H^2(X,\R)$ be given by $(c_1, \cdots, c_M) \in \R^M$, where we assume without loss of generality, that $c_i \geq  0$ for all $i$ (by choosing
	 appropriate identifications with the co-ordinate $\R$ factors).
	 Let $\omega_1, \cdots, \omega_M$ be closed stratified smooth 2-forms on $X$, such that
	 \begin{enumerate}
	 \item $\om =\sum_{i} \om_i$,
	 \item $[\Pi^\ast \om] \in H^2(X,\R)=\R^M$ is given by  $ (0, \cdots, 0, c_i, 0, \cdots,0)$, i.e.\ $c_i$ in the $i-$th place and zero elsewhere.
	  \end{enumerate}
	 
	 Let $\C P^n_i, i=1, \cdots, M$ denote $M$ copies of $\C P^n$, and let $\eta_i$ denote the K\"ahler form on $\C P^n_i$.
	Then, by Theorem \ref{thm-gt},
	 there exists $n \geq 1$ and   embeddings $f_i: X \to \C P^n_i$ such
	that $c_if_i^\ast \eta_i = \om_i$. (This follows from the fact that $\frac{1}{c_i} \om_i$ is integral and by choosing $n$ sufficiently large.)

	Next, let $\C P_\pi = (\C P^n)^M$ denote the product of $M$ copies of $\C P^n$, let $\pi_i:
	\C P_\pi \to \C P^n$ denote the projection to the $i-$th factor, and let
	$\theta_i = \pi_i^* \eta_i$. Let $f_\pi : X \to \C P_\pi$ be given by $f_\pi = (f_1, \cdots , f_M)$.
	Then 
	
	\begin{enumerate}
	\item $f_\pi : X \to \C P_\pi$  is an embedding.
	\item $f_\pi^*(\sum_i c_i \theta_i) = \sum_i \om_i = \om$.
	\item $\Theta = \sum_i c_i \theta_i$ is a symplectic form on $\C P_\pi$.
	\end{enumerate}
Let $\{\cdot \, , \, \cdot\}_\Theta$ denote the Poisson structure on $C^{\infty}(\C P_\pi)$ given by the symplectic form $\Theta$. Then $\{\cdot \, , \, \cdot\}_\Theta$ induces a nondegenerate (centerless) Poisson structure on  $C^{\infty}(\OO p (X))$.
\end{proof}

\subsection{Poisson stratified symplectic structures on symplectic reductions are geometric}

Recall the setup of Proposition \ref{prop-mu}:
\begin{enumerate}
	\item  $(M,\om)$ is a smooth symplectic manifold equipped with a Hamiltonian $G-$action with finitely many orbit types. 
	\item $\mu: M \to \mathfrak{g}^*$ is the associated moment map, and $Z=\mu^{-1}(0)$.
	\item $\phi: M \to \R^n$ is a $G-$equivariant smooth embedding of $M$ into $\R^n$ equipped with a linear $G-$action. \item  $\rho: \R^n/G \to \R^l$ is the Hilbert embedding  and  $X= \rho ( \phi (Z)/G)$ is the embedded image of $M \sslash G = Z/G$ in $\R^l$.
	\item $C^\infty(X)$ is the algebra of restrictions of $C^\infty(\R^l)$ to 
	$X$.
\end{enumerate}

Then, by Proposition \ref{prop-mu},  $C^\infty(X)$ is isomorphic to the algebra $C^\infty(Z)^G$ of $G-$invariant smooth functions on $Z$ (as a subset of $M$).
Hence, the structure in \cite[Section 6]{SL_stratsympred} can be interpreted as a Poisson-symplectic stratified structure on the Whitney stratified space $(X,\Sigma)$ embedded in $\R^l$.
Thus, by \cite{SL_stratsympred},  there exists a collection of symplectic forms $\{\om_S: S \in \Sigma\}$  inducing a  Poisson-symplectic stratified structure as in Definition \ref{def-alsss}. Hence, 
by Proposition \ref{prop-alimpliesgeo}, there exists a smooth 2-form $\theta$ on $\R^l$ such that 
$i_S^*\theta = \mathrm{Inv}_g(\omega_S)$ for all $S \in \Sigma$, where $i_S : S \hookrightarrow \R^l$ is the inclusion map, and $g$ denotes the Euclidean inner product on $\R^l$.

\section{Universal stratified symplectic reduction}

In \cite{gotay-tuynman-sympred,gotay-tuynman-mp}, Gotay and Tuynman proved, what in hindsight, is really an analog of the Gromov-Tischler theorem in the context of reductions. They argued that  in the symplectic category, this is a more natural construction than embeddings. We begin with their definition of reduction from \cite{gotay-tuynman-sympred}:

\begin{defn}\cite{gotay-tuynman-sympred} Let $(M, \omega)$ be a symplectic manifold and $N \subset M$ be a submanifold. Suppose $\omega_N := \iota_N^* \omega$ has constant rank on $N$. Further assume that the characteristic distribution $\ker \omega_N \subset TN$ is \textit{fibrating}, i.e.\
	\begin{enumerate}
	\item  $\ker \omega_N \subset TN$  integrable with associated foliation $\mathscr{F}$, 
	\item the leaf space $M/\mathscr{F}$ is a (Hausdorff) smooth manifold,
	\item $M \to M/\mathscr{F}$ is a smooth fiber bundle.
	\end{enumerate}  
 Then $M/\mathscr{F}$ is called the \textit{reduction of $M$ by $N$}.\end{defn}

The main point of this definition is that it is a purely topological version of Marsden-Weinstein reduction, where $N \subset M$ turns out to be a generic level set of the moment map of some Hamiltonian Lie group action on $M$. With this definition at hand, \cite{gotay-tuynman-sympred} proves that $\Bbb R^{2n}$ equipped with the standard symplectic form $\omega_n$ is the universal symplectic manifold for reduction in the following sense:

\begin{thm}\cite{gotay-tuynman-sympred} Let $(M, \omega)$ be a symplectic manifold such that $[\omega] \in H^2(M; \Bbb R)$ lies in the image of $H^2(M; \Bbb Z)  \to H^2(M; \Bbb R)$. Then there exists $n \geq 1$ such that $(M, \omega)$ can be realized as a reduction of $(\Bbb R^{2n}, \omega_n)$.\end{thm}

The hypothesis of the theorem is easily seen to be satisfied for any compact symplectic manifold. We give a generalization of the theorem in the setup of compact symplectic stratified spaces. Our proof follows a different route than the one taken by Gotay and Tuynman, and uses ingredients of  Theorem \ref{thm-gt} and Section \ref{sec-ngt} in a crucial way.

\begin{thm}\label{gotay}  Let $(X, \omega)$ be a compact symplectic stratified space. Then there exists $n \geq 1$ and a Whitney stratified set $(Y, \Sigma) \subset \Bbb R^{2n}$ equipped with a map $f : Y \to X$ such that 
\begin{enumerate}
\item $(Y, X, f)$ is a fiber bundle with smooth manifold fibers.
\item For every stratum $S \in \Sigma$, $\ker \iota^*_S \omega_n = \ker df|_{S}$
\end{enumerate}
\end{thm}

\begin{proof}Let $(X, \Sigma_X, \mathscr{T}_X)$ denote the underlying abstractly stratified space of the symplectic stratified space $(X, \omega)$. Let $\Pi_X$ denote the compression map for $X$  occurring in Proposition \ref{form-compress}.

First, let us assume $[\Pi_X^* \omega] \in H^2_{sing}(X; \Bbb Z) \subset H^2_{sing}(X; \Bbb R)$, i.e.\ $\omega$ is integral. There exists a complex line bundle $L$ over $X$ corresponding to the class $[\Pi_X^* \omega] \in H^2(X; \Bbb Z)$; let $Y$ be the associated principal $U(1)$-bundle corresponding to $L$, and $f : Y \to X$ be the bundle projection. We equip $Y$ with an abstract stratification by defining the collection of strata as $\Sigma_Y := \{f^{-1}(S) : S \in \Sigma_X\}$ and the tube system as simply a pullback $\mathscr{T}_Y := f^* \mathscr{T}_X$ (see Remark \ref{strat-pullback2}). 

 Let $\Pi_Y$ denote the compression map for $Y$  occurring in Proposition \ref{form-compress}.
As $\Pi_X^*\omega \in \Omega^2_{co}(X)$ is a controlled differential form, we have that $f^*\Pi_X^*\omega = \Pi_Y^* f^* \omega \in \Omega^2_{co}(Y)$ is also a controlled differential form with respect to the abstract stratification on $Y$ defined in the previous paragraph. Then $[\Pi_Y^* f^* \omega] = f^*[\Pi_X^*\omega] \in H^2_{sing}(X; \Bbb R)$ (by functoriality of the isomorphism in Verona's Theorem \ref{thm-veronadr}). However as $f : Y \to X$ is the $U(1)$-bundle with Chern class $[\Pi_X^*\omega]$, $f^*[\Pi_X^*\omega] = 0$. Therefore, by Corollary \ref{cor-0impliesrn}, we obtain an embedding $h : Y \to \Bbb R^{2n}$ such that $h^*\omega_n = f^*\omega$.

Let us identify $h(Y) \subset \Bbb R^{2n}$ with $Y$, so that $f^*\omega$ is simply the restriction of $\omega_n$ to the Whitney stratified subset $Y \subset \Bbb R^{2n}$. Then $f^*\omega$ is  degenerate only along the tangent spaces to the circle fibers of $f : Y \to X$, therefore we conclude the statement of the theorem.

Finally, we indicate the modifications necessary to handle nonintegral forms. The idea is similar to that in Proposition \ref{prop-geoimpliesal}.
 We can, in general, write $\omega = \sum_1^k c_i\omega_i$, where each $\omega_i$ is integral, and $c_i \in \R$ . Let $L_i$ denote the complex line bundle  over $X$ corresponding to the class $[\Pi_X^* \omega_i] \in H^2(X; \Bbb Z)$. Let 
 $L=\oplus_1^k L_i$, and $Y$ be the associated principal $\times_i^k \big(U(1)\big)_i$-bundle corresponding to $L$.
The rest of the argument is identical.
\end{proof}


\begin{thebibliography}{GWPL76}
		\bibitem[CdS10]{cds}
Ana	Cannas da Silva. 
\newblock  Fold-forms for four-folds. 
	\newblock {\em J. Symplectic Geom.} 8, no. 2, 189--203, 2010. 
	
	\bibitem[DJ96]{DJ_singsymp}
	W.~Domitrz and S.~Janeczko.
	\newblock On {M}artinet's singular symplectic structures.
	\newblock In {\em Singularities and differential equations ({W}arsaw, 1993)},
	volume~33 of {\em Banach Center Publ.}, pages 51--59. Polish Acad. Sci. Inst.
	Math., Warsaw, 1996.
	
	\bibitem[EM02]{em-book}
	Y.~Eliashberg and N.~Mishachev.
	\newblock {\em Introduction to the {$h$}-principle}, volume~48 of {\em Graduate
		Studies in Mathematics}.
	\newblock American Mathematical Society, Providence, RI, 2002.
	
	

	\bibitem[GM88]{GM_SMT}
	Mark Goresky and Robert MacPherson.
	\newblock {\em Stratified {M}orse theory}, volume~14 of {\em Ergebnisse der
		Mathematik und ihrer Grenzgebiete (3) [Results in Mathematics and Related
		Areas (3)]}.
	\newblock Springer-Verlag, Berlin, 1988.
	
	\bibitem[GMP14]{gmp}
Victor	Guillemin,  Eva Miranda and Ana Rita Pires. 
	\newblock Symplectic and Poisson geometry on $b$-manifolds.
	\newblock {\em 	 Adv. Math.} 264, 864--896, 2014.
	
		\bibitem[Gor76]{goresky-thesis}
			R.~Mark Goresky.
	\newblock  Geometric cohomology and homology of stratified objects.
		\newblock {\em Ph.D. Thesis, Brown Univ.}, Providence, R.I., 1976
	
	
	\bibitem[Gor78]{Goresky_lines}
	R.~Mark Goresky.
	\newblock Triangulation of stratified objects.
	\newblock {\em Proc. Amer. Math. Soc.}, 72(1):193--200, 1978.
	
	\bibitem[Gro71]{gromov-icm70}
	M.~L. Gromov.
	\newblock A topological technique for the construction of solutions of
	differential equations and inequalities.
	\newblock In {\em Actes du {C}ongr\`es {I}nternational des {M}ath\'{e}maticiens
		({N}ice, 1970), {T}ome 2}, pages 221--225. 1971.
	
	\bibitem[Gro86]{Gromov_PDR}
	Mikhael Gromov.
	\newblock {\em Partial differential relations}, volume~9 of {\em Ergebnisse der
		Mathematik und ihrer Grenzgebiete (3) [Results in Mathematics and Related
		Areas (3)]}.
	\newblock Springer-Verlag, Berlin, 1986.
	
	\bibitem[GT89]{gotay-tuynman-sympred}
	Mark~J. Gotay and Gijs~M. Tuynman.
	\newblock {${\bf R}^{2n}$} is a universal symplectic manifold for reduction.
	\newblock {\em Lett. Math. Phys.}, 18(1):55--59, 1989.
	
	\bibitem[GT91]{gotay-tuynman-mp}
	M.~J. Gotay and G.~M. Tuynman.
	\newblock A symplectic analogue of the {M}ostow-{P}alais theorem.
	\newblock In {\em Symplectic geometry, groupoids, and integrable systems
		({B}erkeley, {CA}, 1989)}, volume~20 of {\em Math. Sci. Res. Inst. Publ.},
	pages 173--182. Springer, New York, 1991.
	
	
	\bibitem[GWPL76]{gwpl}
Christopher G.	Gibson, Klaus Wirthm\"{u}ller, Andrew A. du Plessis, and Eduard J. N. Looijenga. 
	\newblock Topological stability of smooth mappings.
	\newblock {\em	 Lecture Notes in Mathematics}, Vol. 552. Springer-Verlag, Berlin-New York, 1976. {\rm iv}+155 pp. 
	
	
	\bibitem[HP19]{herbig-pflaum}
	Hans-Christian Herbig and Markus~J. Pflaum.
	\newblock Invariant {W}hitney functions.
	\newblock {\em Internat. J. Math.}, 30(2):1950009, 26, 2019.
	
	\bibitem[KL00]{kutz-loose}
	Frank Kutzschebauch and Frank Loose.
	\newblock Real analytic structures on a symplectic manifold.
	\newblock {\em Proc. Amer. Math. Soc.}, 128(10):3009--3016, 2000.
	
	\bibitem[Ler05]{lermann-expo}
	Eugene Lerman.
	\newblock Gradient flow of the norm squared of a moment map.
	\newblock {\em Enseign. Math. (2)}, 51(1-2):117--127, 2005.
	
	\bibitem[LMS93]{sjamaar-lms}
	Eugene Lerman, Richard Montgomery, and Reyer Sjamaar.
	\newblock Examples of singular reduction.
	\newblock In {\em Symplectic geometry}, volume 192 of {\em London Math. Soc.
		Lecture Note Ser.}, pages 127--155. Cambridge Univ. Press, Cambridge, 1993.
	
	\bibitem[Mal67]{malgrange-bk}
	B.~Malgrange.
	\newblock {\em Ideals of differentiable functions}, volume~3 of {\em Tata
		Institute of Fundamental Research Studies in Mathematics}.
	\newblock Tata Institute of Fundamental Research, Bombay; Oxford University
	Press, London, 1967.
	
	\bibitem[Mat77]{mather-diffinv}
	John~N. Mather.
	\newblock Differentiable invariants.
	\newblock {\em Topology}, 16(2):145--155, 1977.
	
	\bibitem[Mat12]{Mather_notes}
	John Mather.
	\newblock Notes on topological stability.
	\newblock {\em Bull. Amer. Math. Soc. (N.S.)}, 49(4):475--506, 2012.
	
	
	\bibitem[Mos57]{mostow-mp}
	George D. Mostow.
	\newblock  Equivariant embeddings in Euclidean space, 
	\newblock {\em Annals of Mathematics, Second Series}, 65: 432--446, 1957.
		
		
		
	\bibitem[MS99]{sjamaar-rednquant}
	Eckhard Meinrenken and Reyer Sjamaar.
	\newblock Singular reduction and quantization.
	\newblock {\em Topology}, 38(4):699--762, 1999.
	
	\bibitem[MS22]{ms-hprin}
	M.~Mj and B.~Sen.
	\newblock $h$-Principle for Stratified Spaces
	\newblock {\em Int. Math. Res. Not. IMRN}, 2023,
	\newblock \url{https://doi.org/10.1093/imrn/rnad123}.
	
	\bibitem[Nat80]{Natsume_whitney}
	Hiroko Natsume.
	\newblock The realization of abstract stratified sets.
	\newblock {\em Kodai Math. J.}, 3(1):1--7, 1980.
	
		\bibitem[Pal57]{palais-mp}
	Richard S. Palais. 
	\newblock Imbedding of compact, differentiable transformation groups in orthogonal representations, 
	\newblock {\em  Journal of Mathematics and Mechanics}, 6: 673--678,  1957.
	
	
	\bibitem[Pfl01]{pflaum-expo}
	Markus~J. Pflaum.
	\newblock Smooth structures on stratified spaces.
	\newblock In {\em Quantization of singular symplectic quotients}, volume 198 of
	{\em Progr. Math.}, pages 231--258. Birkh\"{a}user, Basel, 2001.
	
	
	\bibitem[Sch75]{schwarz-Rn}
	Gerald~W. Schwarz.
	\newblock Smooth functions invariant under the action of a compact {L}ie group.
	\newblock {\em Topology}, 14:63--68, 1975.
	
	\bibitem[SL91]{SL_stratsympred}
	Reyer Sjamaar and Eugene Lerman.
	\newblock Stratified symplectic spaces and reduction.
	\newblock {\em Ann. of Math. (2)}, 134(2):375--422, 1991.
	
	\bibitem[SLV15]{slv}
	Petr Somberg, H\^{o}ng~V\^{a}n L\^{e}, and Ji\v{r}i Van\v{z}ura.
	\newblock Poisson smooth structures on stratified symplectic spaces.
	\newblock In {\em Mathematics in the 21st century}, volume~98 of {\em Springer
		Proc. Math. Stat.}, pages 181--204. Springer, Basel, 2015.
	
	
	\bibitem[Teu81]{teu}
	Michael Teufel. 
		\newblock  Abstract prestratified sets are $(b)$-regular. 
		\newblock {\em	J. Differential Geometry,} 16 no. 3, 529--536,  1981. 
	
	\bibitem[Tho69]{Thom_stratmaps}
	Ren\'{e} Thom.
	\newblock Ensembles et morphismes stratifi\'{e}s.
	\newblock {\em Bull. Amer. Math. Soc.}, 75:240--284, 1969.
	
	\bibitem[Tis77]{tischler}
	David Tischler.
	\newblock Closed {$2$}-forms and an embedding theorem for symplectic manifolds.
	\newblock {\em J. Differential Geometry}, 12(2):229--235, 1977.
	
	\bibitem[Ver71]{Verona_derham}
	Andrei Verona.
	\newblock Le th\'{e}or\`eme de de {R}ham pour les pr\'{e}stratifications
	abstraites.
	\newblock {\em C. R. Acad. Sci. Paris S\'{e}r. A-B}, 273:A886--A889, 1971.
	
\end{thebibliography}
\end{document}